\documentclass[11pt,letterpaper]{article}
\usepackage{fullpage}
\usepackage{amsmath}
\usepackage{amsfonts}
\usepackage{amssymb}
\usepackage{tikz-cd} 
\usepackage{titlesec}
\usepackage{tasks}
\usepackage{mathabx}
\usepackage{color}
\usepackage{mathrsfs}
\usepackage{hyperref}
\usepackage[all]{xy}
\usepackage{graphicx}
\usepackage{cite}

\title{K-Theoretic $I$-function of $V//_{\theta} \mathbf{G}$ and Application}
\author{Yaoxiong Wen}
\date{}

\newtheorem{theorem}{Theorem}[section]
\newtheorem{lemma}{Lemma}[section]
\newtheorem{proposition}{Proposition}[section]
\newtheorem{corollary}{Corollary}[section]
\newtheorem{definition}{Definition}[section]
\newtheorem{example}{Example}[section]

\newenvironment{proof}{{\noindent\it Proof}\quad}{\hfill $\square$\par}

\newenvironment{remark}[1][Remark]{\begin{trivlist}
\item[\hskip \labelsep {\bfseries #1}]}{\end{trivlist}}

\newcommand{\qed}{\nobreak \ifvmode \relax \elseL
      \ifdim\lastskip<1.5em \hskip-\lastskip
      \hskip1.5em plus0em minus0.5em \fi \nobreak
      \vrule height0.75em width0.5em depth0.25em\fi}
\begin{document}
\maketitle
\begin{abstract}
In this paper, we compute K-theoretic $I$-function with level structure (defined by quasi-map theory) of GIT-quotient of a vector space via abelian and non-abelian correspondence. As a consequence,  we generalize Givental-Lee's result to find the analogous ``Toda operators''  for the $I$-function with nontrivial level structures in the case of complete flag variety. 
\end{abstract}

\tableofcontents
\clearpage

\section{Introduction}
Recently, Ruan-Zhang \cite{2018arXiv180406552R} introduce the level structures and there is a serendipitous discovery that some special target spaces with certain level structures result in Mock theta functions, so it's interesting to see whether the level structures will bring some new structures in quantum K theory \cite{lee2004quantum} \cite{givental2000wdvv}.

Let $X$ be a GIT quotient $V//\mathbf{G}$ where $V$ is a vector space and $\mathbf{G}$ is a connected reductive complex Lie group. The theory of the moduli space of quasimaps to GIT is established in \cite{2011arXiv1106.3724C} where Ciocan-Fontanine, Kim and Maulik define the big $I$-function, in the following paper \cite{2013arXiv1304.7056C}. The first two authors prove the wall-crossing formulas which relate big $I$-function and Givental's big $J$-function of $V//\mathbf{G}$. The K-theoretic stable quasimaps invariants are defined by Tseng-You in \cite{2016arXiv160206494T}. We compute the small so called $I$-function with levels which is defined in terms of certain localization residues on a moduli space of quasimaps:
\begin{align*}
{I}^{R,l}(q,Q):= \mathcal { J } _ { S _ { \infty } } ^ { R , l , 0^{+} } ( 0 , Q ) =  1 + \sum _ { \beta \geq 0 } Q ^ { \beta } (ev_{\bullet})_{*} \left(  \mathcal { O } _ { \operatorname{F} _ { 0 , \beta } } ^ { \mathrm { vir } } \otimes  \left( \frac { \operatorname { t r } _ { \mathbb { C } ^ { * } } \mathcal { D } ^ { R , l } } { \lambda_{-1}^{\mathbb{C}^*}  N _ { \operatorname{F} _ { 0 , \beta } } ^ { \vee } } \right) \right)  
\end{align*}
Abelian/non-abelian correspondence for $I$-function with level 0 is established recently by Rachel \cite{webb2018abelian}, we follow her paper and work on $I$-function with levels and we arrive at the following K theoretic abelian/non-abelian
correspondence:
\begin{theorem}
Let $\mathbf{G}$ be a connected reductive complex Lie group with character $\theta$, acting on a vector space $V$. Then 
\begin{align*}
\psi^*{I}_{\beta}^{V//\mathbf{G},R,l}(q,Q)
=j^* \sum_{W\tilde{\beta} \rightarrow \beta} \sum_{w \in W/W_L} w \left[ \prod_{\alpha} \frac{\prod _ {k=-\infty}^{\tilde{\beta}(\alpha)}(1-L^{\vee}_{\alpha}q^k)}{\prod _{k=-\infty}^{0}(1-L^{\vee}_{\alpha}q^k)}  I_{\tilde{\beta}}^{V//\mathbf{T},R,l} \right]
\end{align*}
where $j$ is an open immersion induced by the inclusion $ V^s(\mathbf{G}) \subset V^s(\mathbf{T})$. 
\end{theorem}

A version of mirror theorem that the above $I$-function lies on the Givental Lagrangian cone of stable map theory is certainly expected. However, we do not know how to prove it at such generality.
Instead, we focus on the theory with level zero on partial flag variety $F l _ { r _ { 1 } , \ldots , r _ { I } } \left( \mathbb { C } ^ { n } \right)$  parameterizing inclusion sequences  $V_{\bullet} : V_1 \hookrightarrow \cdots \hookrightarrow V_I $ of planes $ V_i $ in $ \mathbb{C}^n $ of dimension $ r_i $, it has a GIT representation 
\begin{align*}
F l _ { r _ { 1 } , \ldots , r _ { I } } \left( \mathbb { C } ^ { n } \right) =  \oplus^I_{i=1} \mathrm{Hom}(\mathbb{C}^{r_i},\mathbb{C}^{r_{i+1}}) //_{\mathbf{det}} GL_{r_1} \times \cdots \times GL_{r_I}
\end{align*}
Using explicit geometry (Section 2.6), we can get around the difficulty to show  the following theorem:
\begin{theorem}
The small $I$-function with level 0 equals to the small $\mathcal{J}$-function $\mathcal{J}(0) $.
\end{theorem}
 
 During the course of this work, we note an article \cite{taipale2013k} where the author claimed a version of mirror theorem with a different $I$-function.
 
If $X$ is a complete flag variety, Givental and Y.P. Lee \cite{givental2003quantum} prove  that the $K_G^* (X)$-valued vector-series $p^{\frac{ln Q}{ ln q}}\mathcal{J}(Q, q)$ is the eigen-vector of the finite-difference Toda operator $\hat{H}_{Q,q} \coloneq q ^ {  Q_1  \partial_{ Q _ { 1 } } }+ q ^ { Q_2 \partial_{ Q _ { 2 } }- Q_1 \partial_{ Q _ { 1 } }} \left( 1 - { Q_ 1 } \right) + \ldots + q ^ { - Q_r \partial_{ Q _ { r } }} \left( 1 - { Q_r } \right)$ with the eigen-value $\Lambda^{-1}_0  + \cdots + \Lambda^{-1}_r $. It is surprising that certain special property (Section 4 (6)) of small $J$-function (hence small $I$-function by the theorem) enable us to generalize their result to the case of nontrivial level.
\begin{corollary}
The $K_{\mathbf{T}}(X)$-valued vector-series $\tilde{\mathcal{J}}^{\theta_i,l_i,0^+}(Q,q)$ is the eigen-vector of the finite-difference operator $\widetilde{H}^{\theta_i,l_i}_{Q,q}$ with the eigen-value $\Lambda _ { 0 } ^ { - 1 } + \ldots + \Lambda _ { r } ^ { - 1 }$. Where
\begin{align*}
\widetilde{H}^{\theta_i,l_i}_{Q,q} = q ^ {  Q_1  \partial_{ Q _ { 1 } } }+ \cdots + q ^ { Q_{i+1} \partial_{ Q _ { i+1 } } - Q_{i} \partial_{ Q _ { i } }} \left( 1 - { Q_ i } \circ q^{l_i Q_i \partial_{Q_i}} \right) + \ldots + q ^ { - Q_r \partial_{ Q _ { r } }} \left( 1 - { Q_r } \right)
\end{align*}
\end{corollary}
and $\tilde{\mathcal{J}}^{\theta_i,l_i,0^+}(Q,q)$ is the small $\mathcal{J}$-function with nontrivial level structures.

This paper is arranged as follows. Section 2 revisits K-theoretic quasi-map theory in which we review some basic definitions and theorems. In section 3, we obtain the K theoretic abelian and non-abelian correspondence for $I$-function with level structure. Finally, we apply this mirror theorem to complete flag variety and find difference equations for complete flag variety with level structures.

\section*{Acknowledgments}
I would like to thank Prof.Yongbin Ruan for suggesting this problem and for useful discussions. Thanks are also due to  Ming Zhang for discussion about level structures and many enlightening suggestions, and Rachel Webb for discussion about abelian/non-abelian correspondence for $I$-function, and Zijun Zhou and Jeongseok Oh for helpful discussions. This article is partially supported by the China Scholarship Council grant No.201706010023.

\section{A Review of  K-theoretic Quasimap invariants}
This section mainly follows \cite{manolache2014stable} \cite{2018arXiv180406552R} \cite{2016arXiv160206494T} . Here we focus on stable quasimaps to GIT quotient $ V//\mathbf{G} $, where $V$ is a finite dimensional vector space, and $\mathbf{G} $ is a connected reductive complex algebraic group. Let's first recall some basic definitions.

\subsection{GIT quotients}
Suppose $ V = \operatorname { Spec } ( A ) $ is an affine space with a reductive group $ \mathbf{G} $ action. Let $ \chi(\mathbf{G}) $ be the character group of $\mathbf{G}$, i.e. $  \forall \alpha \in \chi(\mathbf{G}), \alpha : \mathbf{G} \rightarrow \mathbb{C}^* $  and $ \operatorname{Pic}^{\mathbf{G}}(V) $ be the group of isomorphism classes of $\mathbf{G}$-linearized line bundles on $V$. Fix a character $ \theta \in \chi(\mathbf{G}) $, consider a $\mathbf{G}$-linearized line bundle 
\begin{align*}
L_{\theta} = V \times \mathbb{C}_{\theta}
\end{align*}
The group $ \mathbf{G} $ acts on the line bundle $L_{\theta}$ by $ g \cdot (x , z ) = ( gx , \chi^{-1}(g)z )  $, so the $ \mathbf{G} $ invariant section of $ L_{\theta}^n $ is a function $ f(x)z^n $ on $ \mathbb{C}[V \times \mathbb{C} ] $, where $ f(x) \in \mathbb{C}[V] $ is a relative invariant of weight $ \chi^n $, i.e. $ f(gx) = \chi^n(g)f(x) $. In order to define GIT quotient, we need to introduce two terms.
\begin{definition}
$( \mathrm { i } )$ A point $v \in V$ is called $\theta$-semistable if there exist a relative invariant function  $f(x)$  of weight $n$, for some  $n \geq 1$ , such that  $f(v) \neq 0$. The set of $\theta$-semistable points in $V$ is denoted by $V^{ss}_{\theta}$. \\
$( \mathrm { ii } )$ A point $v \in V$ is called $\theta$-stable if it is $\theta$-semistable, further, $ \operatorname { dim } \mathbf{G} \cdot x = \operatorname { dim } \mathbf{G} / \Delta $ where $\Delta $ is the kernel of the representation, so this condition  means that the stabilizer of $x$ has zero dimention and the action of $ \mathbf{G} $ is closed. The set of $\theta$-stable points in $V$ is denoted by $V^{s}_{\theta}$.
\end{definition}

Since $\mathbf{G}$ is reductive, the graded algebra of $ \mathbf{G} $ invariant sections
\begin{align*}
S(L_{\theta} ) : = \oplus_{n \geq 0 } \Gamma(V, L^n_{\theta} )^{\mathbf{G}}
\end{align*}
is finitely generated, so the associate GIT quotient
\begin{align*}
V//\mathbf{G} := V//_{\theta} \mathbf{G} = \operatorname{Proj} S(L_{\theta} ) = \operatorname{Proj}( \oplus_{n \geq 0 } \Gamma(V, L^n_{\theta} )^{\mathbf{G}} ) 
\end{align*}
is a quasiprojective variety. 
\begin{remark}
Nagata \cite{nagata1963invariants} also showed that if a reductive linear algebraic group $\mathbf{G}$ acts on an affine variety $X$, then the invariant subalgebra is finitely generated 
\end{remark}

Geometrically, this quotient can be viewed as a quotient of open set $ V^{ss}_{\theta} $ under the equivalence relation $ x \sim y $ if and only if $ \overline{\mathbf{G} \cdot x}   $ intersects $ \overline{\mathbf{G} \cdot y} $.

In the following discussion, we assume 

( 1 ) $ \emptyset \neq V ^ { s } = V ^ { s s }  $

( 2 ) $ V ^ { s }$  is nonsingular.   

( 3 ) $ \mathbf{G}$ acts freely on  $V ^ { s } $ 

Under this assumption, the orbits $ {\mathbf{G} \cdot x} $ and $ {\mathbf{G} \cdot y} $ are already closed by stable condition, so $ V^{s} $ is a principal $\mathbf{G}$ bundle over $ V//_{\theta} \mathbf{G} $ in the \'etale topology, hence $ V//_{\theta} \mathbf{G} $ coincides with the quotient stack $ [ V^s /\mathbf{G} ] $. Usually, we use Mumford's Numerical Criterion to find the semistable and stable points:
\begin{proposition}\cite{King94moduliof}
A point $x \in V$ is $ \theta $-semistable if and only if $ \theta(\Delta)={1} $ and every one parameter subgroup $ \lambda $ of $ \mathbf{G} $, for which $ \mathrm{lim}_{t \rightarrow 0} \lambda(t) \cdot x $ exists, satisfies $ \langle \theta , \lambda \rangle \geqslant 0 $. Such a point is $\theta$-stable if and only if the only one-parameter subgroups $\lambda$ of $\mathbf{G}$, for which $ \mathrm{lim}_{t \rightarrow 0} \cdot x $ exists and $ \langle \theta , \lambda \rangle = 0 $, are in $ \Delta $.
\end{proposition}
\begin{remark}
Here $ \langle \theta , \lambda \rangle $ means that if $ \theta(\lambda(t))= t^m $ the we denote $ \langle \theta , \lambda \rangle = m $.
\end{remark}
\begin{example}{Projective space} 
$\mathbb{P}^n = (\mathbb{C}^{n+1} \setminus \{0 \} )/\mathbb{C}^{*} $ where $ t \cdot (x_0, \cdots , x_n ) = (t x_0, \cdots , t x_n ) $ for $t \in \mathbb{C}^{*}$, $ (x_0, \cdots , x_n ) \in \mathbb{P}^n$. If we take $ \theta : \mathbb{C}^{*} \rightarrow \mathbb{C}^{*} $ to be $ \theta(t) = t $, then the $ \mathbb{C}^{*} $ invariant section of $ L _ { \theta } = \mathbb{C}^{n+1} \times \mathbb { C } _ { \theta } $ is just the homogenous function of degree $n$, then $ \mathbb{C}^{n+1}//_\theta \mathbb{C}^* = \operatorname { Proj } \left( \oplus _ { n \geqslant 0 } \Gamma \left( \mathbb{C}^{n+1} , L _ { \theta } ^ { n } \right) ^ { \mathbb { C }^* } \right) $ which is just the projective space.

Alternatively, we use Mumford's Numerical Criterion to find the stable loci. Since every one parameter subgroup of $ \mathbb{C}^* $ is a map $ \lambda : \mathbb{C}^* \rightarrow \mathbb{C}^* $, such that, $ \lambda(t)=t^m $, for some $ m \in \mathbb{Z} $, so for $ x \in \mathbb{C}^{n+1} $, if $ x \neq 0 $, then $ \lambda(t)(x)=t^mx $ exists if and only if $ m \geq 1 $, but if $ x=0 $, then $ \lambda(t)(x)=t^mx $ always exists for $ \forall m \in \mathbb{Z} $, so $\{0\} \in \mathbb{C}^{n+1} $ is the unstable point, then $ \mathbb{C}^{n+1}//_\theta \mathbb{C}^* =  (\mathbb{C}^{n+1} \setminus \{0 \} )/\mathbb{C}^{*} $.
\end{example}

\subsection{Moduli space of  $\epsilon$-stable quasimaps to $V//_{\theta}\mathbf{G}$}
Let $ (C,x_1, \cdots, x_n ) $ be a prestable pointed genus $g$ curve, then a map from curve $C$ to quotient stack $ [ V/\mathbf{G} ] $ is a combination of following maps
\begin{align*}
C \stackrel { u } { \rightarrow } \mathcal{P} \times _ { \mathbf { G } } V \rightarrow [ V / \mathbf { G } ]
\end{align*}
denoted by a pair $ (\mathcal{P},{u}) $ where $\mathcal{P}$ is a principal $ \mathbf{G} $ bundle on $C$ and $\mathcal{P} \times _ { \mathbf { G } } V$ means the diagonal action (or balanced action) of $ \mathbf{G} $, i.e. $ g \cdot (p,v) = (gp,g^{-1}v ) $. Since $ \operatorname { Pic } ^ { \mathbf { G } } ( V ) = \operatorname { Pic } ( [ V / \mathbf { G } ] ) $, then we define degree of the map as follows
\begin{definition}
The degree $\beta$ of $(\mathcal{P},u)$ is a homomorphism
\begin{align*}
\beta : \operatorname { Pic } ^ { \mathbf { G } } ( V ) \rightarrow \mathbb{Z} ,\ \
\beta(L) = \operatorname{deg}_{C}(u^*(\mathcal{P} \times _ { \mathbf { G } } L ) )
\end{align*}
\end{definition}

\begin{lemma}
If $(C, \mathcal{P}, u)$ is a quasimap, then $\beta(L_\theta) \geq 0$, and $\beta(L_\theta) = 0$ if and only if $\beta = 0$, if and only if the quasimap is a constant map to $V//\mathbf{G}$.
\end{lemma}
We call $\beta \in \mathrm { Hom }_{\mathbb{Z}} \left( \mathrm { Pic } ^ { \mathrm { G } } ( W ) , \mathbb { Z } \right)$ which realized as classes of quasimaps to $V / /_\theta \mathbf { G }$ the $L_{\theta}$-effective classes, they form a semigroup, denoted $\operatorname { Eff } ( V , \mathbf { G } , \theta )$.

\begin{definition}
An n-pointed, genus $g$ quasimap of class $\beta$ to $ V//\mathbf{G} $ consists of the data $ ((C,x_1, \cdots, x_n), \mathcal{P}, u ) $, besides, there is a finite set $B \subset C$ called base points, such that $ \forall p \in C \backslash B $, $ u(p) \in V^s $. 

We call a quasimap prestable, if the base points are away from nodes and markings on $C$, and stable, if the line bundle $\omega _ { C } \left( \sum _ { i = 1 } ^ { n } x _ { i } \right) \otimes \mathcal { L } ^ { \epsilon }$ is ample for every rational number $\epsilon > 0 $, where
$$  \mathcal { L } := u ^ { * } \left( \mathcal{P} \times _ { \mathbf { G } } L _ { \theta } \right) \cong \mathcal{P} \times _ { \mathbf { G } } \mathbb { C } _ { \theta }   $$ 

We say two quasimaps $\left(( C , x _ { 1 } , \ldots , x _ { n }) , \mathcal{P} , u \right)$ and $\left( (C ^ { \prime } , x _ { 1 } ^ { \prime } , \ldots , x _ { n } ^ { \prime } ), \mathcal{P} ^ { \prime } , u ^ { \prime } \right)$ are isomorphism if there are two isomorphisms $ f : C \rightarrow C^{\prime}$ and $ \sigma : \mathcal{P} \rightarrow f ^ { * } \mathcal{P} ^ { \prime } $ such that
$$  f \left( p _ { j } \right) = p _ { j } ^ { \prime } ,\ \ \sigma _ { V } ( u ) = f ^ { * } \left( u ^ { \prime } \right)  $$  
where $\sigma _ { V } : \mathcal{P} \times _ { \mathbf { G } } V \rightarrow \mathcal{P} ^ { \prime } \times_{ \mathbf { G }} V^{ \prime } $ is the isomorphism of bundles induced by $ \sigma $.
\end{definition}
\begin{remark}[Remark]
(1) For a stable quasimap, every rational component of curve $C$ should have at least two markings or nodes and $\operatorname{deg}(\mathcal{L}) > 0$. \\
(2) The automorphism group of stable quasimap is finite and reductive. 
\end{remark}

Given a prestable quasimap $ \left( \left( C , x _ { 1 } , \ldots , x _ { n } \right) , \mathcal{P} , u \right) $ to $ V//\mathbf{G} $ we define the contact order of $ u(C) $ with the unstable closed subscheme $\mathcal{P} \times _ \mathbf{ G } V ^ { u s }$ at $ u(x) $ as
\begin{definition}
The length $l(x)$ of a point $ x \in C $ is 
$$ l ( x ) : = \operatorname { length } _ { x } \left( \operatorname { coker } \left( u ^ { * } \mathcal { J } \rightarrow \mathcal { O } _ { C } \right) \right) $$
where $ \mathcal { J } $ is the ideal sheaf of $ \mathcal{P} \times _ \mathbf{ G } V ^ { u s } $.
\end{definition} 

Using $ l(x) $ we can define a new stability condition called $ \epsilon$-stability. 
\begin{definition}
Given a positive rational number $ \epsilon $, a quasimap $ \left( \left( C , x _ { 1 } , \ldots , x _ { n } \right) , \mathcal{P} , u \right) $ is called $ \epsilon $-stable if 

 (1) $ \omega _ { C } \left( \sum _ { i = 1 } ^ { k } x _ { i } \right) \otimes \mathcal { L } _ { \theta } ^ { \epsilon }$ is ample.
 
 (2) $ \epsilon l ( x ) \leq 1 $ for every point $x \in C$.
\end{definition}

\begin{remark}[Remark]
Notice that the length $ l(x) $ is an integer, so when $ \epsilon $ is big enough, every point in $C$ will be sent to $ V^{s} $.
\end{remark}
  
\begin{definition} 
The moduli stack of n-pointed $ \epsilon $-stable quasimaps of genus $g$ and degree $ \beta $ to $ V / / \mathbf { G } $ denoted by $ \mathcal { Q }^{ \epsilon } _ { g , k } ( V / / \mathbf { G } , \beta ) $ consists of the following data
$$  \mathcal { Q } ^ { \epsilon } _ { g , k } ( V / / \mathbf { G } , \beta )(S) = \left\{ \pi : \mathcal { C }^{ \epsilon } \rightarrow S , \left\{ x _ { i } : S \rightarrow \mathcal { C }^{ \epsilon } \right\} _ { i = 1 , \ldots , n } , \mathcal { P } , u \right\} $$
where $\pi$ is a flat morphism from family of curves $\mathcal { C }^{\epsilon}$ to base scheme $S$, $p_i$ are sections of $\pi$, $\mathcal { P }$ is the principal $\mathbf{G}$-bundle on $\mathcal { C }^{\epsilon}$ and $u : \mathcal { C }^{\epsilon} \longrightarrow \mathcal { P } \times _ { \mathbf { G } } V $ is a section, such that the restriction to every geometric fiber $\mathcal { C }^{\epsilon}_s$ of $\pi$ is a $\epsilon$-stable n-pointed quasimap of genus $g$ and class $\beta$.
\end{definition}
\begin{theorem}
The stack $ \mathcal { Q }^{ \epsilon } _ { g , k } ( V / / \mathbf { G } , \beta ) $  is a separated Deligne--Mumford stack of finite type, admitting a canonical obstruction theory. If W has at most lci singularities, then the obstruction theory is perfect.
\end{theorem}

\subsection{Quasimap graph space}
Now we can define quasimap graph space, here we only consider the special case in which we can define I-function. Let's take $g=0$ and parametrized component $D=\mathbb{P}^1$.
\begin{definition} 
The moduli stack $ \mathcal { QG }^{\epsilon} _ { 0 , n } ( V // \mathbf { G } , \beta ) $ of n-pointed stable quasimaps of genus 0 and degree $\beta$ to $ V//\mathbf{G}$ with parametrized component $\mathbb{P}^1$ consists of the data
$$ \mathcal { QG }^{\epsilon} _ { 0 , n } ( V // \mathbf { G } , \beta )(S) = \left\{ \pi : \mathcal { C }^{\epsilon} \rightarrow S , \left\{ x _ { i } : S \rightarrow \mathcal { C }^{\epsilon} \right\} _ { i = 1 , \ldots , n } , \mathcal { P } , u, \varphi \right\}  $$
here the only new data is $ \varphi : \mathcal{C}^{\epsilon} \rightarrow \mathbb{P}^1$ such that on every geometric fiber  $\mathcal{C}^{\epsilon}_s = C$ , $ \varphi_s : C \rightarrow \mathbb{P}^1$ is a regular map such that $ \varphi _ { * } [ C ] = [ \mathbb{P}^1 ] $. The stability condition changes to require the line bundle 
$$ \omega _ { C } \left( x _ { 1 } + \cdots + x _ { n } \right) \otimes \mathcal { L } ^ { \epsilon } \otimes \varphi ^ { * } \left( \omega _ { \mathbb{P}^{1} } ^ { - 1 } \otimes \mathcal { M } \right) $$
is ample for $ \forall \epsilon >0 $, where $ \mathcal { M } $ is any ample line bundle on $ \mathbb{P}^1 $. Equivalently, it requires no condition on distinguish component $C_0$ of $C$ that $ C_0 \cong \mathbb{P}^1$. And
$$ \epsilon l ( x ) \leq 1  $$
holds for every point $x \in C$.
\end{definition}
\begin{remark}[Remark]
When $ \epsilon >2 $, we have 
\begin{align*}
\mathcal{Q} _ { g , k } ^ { \infty } ( V / / \mathbf { G } , \beta )   &= \overline { M } _ { g , k } ( V / / \mathbf { G } , \beta ) \\ 
\mathcal{Q G} _ { g , k  } ^ { \infty } ( V / / \mathbf { G }, \beta ) &= \overline { M } _ { g , k } \left( V / / \mathbf { G } \times \mathbb { P } ^ { 1 } , ( \beta , 1 ) \right)
\end{align*}
When $ \epsilon=0^+ $, we get the $\mathcal{Q} _ { g , k } ^ { 0^+ } ( V / / \mathbf { G } , \beta )$ and $ \mathcal{Q G} _ { g , k  } ^ { 0^+ } ( V / / \mathbf { G }, \beta ) $ which are the same as the original ones defined in \cite{manolache2014stable}.
\end{remark}

There are two natural forgetful maps 
\begin{align*}
& \mu : \mathcal{QG} _ { 0 , n } ^ { \epsilon } ( V / / \mathbf{G} , \beta ) \rightarrow \widetilde{\mathbb{P}^1[n]}    \\
& \nu : \mathcal{QG} _ { 0 , n } ^ { \epsilon } ( V / / \mathbf{G} , \beta ) \rightarrow \mathfrak { B } \mathfrak { u } \mathfrak { n } _ { G }
\end{align*}
where $\widetilde{\mathbb{P}^1[n]}$  is the Fulton-MacPherson space of (not necessarily stable) configurations of k distinct points on $\mathbb{P}^1$ and $ \mathfrak { B } \mathfrak { u } \mathfrak { n } _ { G } $ is the relative moduli stack of principal $ \mathbf{G} $-bundle on fibers of universal curve $\mathfrak { C } _ { g , k } \rightarrow \mathfrak { M } _ { g , k }$. They are both smooth Artin stack, locally of finite type.
We have an Euler sequence
\begin{align*}
0 \longrightarrow \mathfrak { P } \times _ { \mathbf { G } } \mathfrak { g } \longrightarrow \mathfrak { P } \times _ { \mathbf { G } } V \longrightarrow \mathcal { F } \longrightarrow 0
\end{align*}
on the universal curve
\begin{align*}
\pi : \mathfrak { C }^{\epsilon} \rightarrow \mathcal { QG }^{\epsilon} _ { 0 , n } ( V // \mathbf { G } , \beta )
\end{align*}
where $ \mathfrak { P } $ is the universal principal $ \mathbf{G} $-bundle on universal curve $ \mathcal { C }^{\epsilon} $. Then the $ \mu $-relative obstruction theory is given by
\begin{align*}
\left( R ^ { \bullet } \pi _ { * } \mathcal { F } \right) ^ { \vee }
\end{align*}
and the $\nu$-relative obstruction theory is given by
\begin{align*}
\left( R ^ { \bullet } \pi _ { * } \left( \mathfrak { P } \times _ { \mathbf { G } } V \right) \right) ^ { \vee }
\end{align*}
\begin{theorem}\cite{2013arXiv1304.7056C}\cite{2016arXiv160206494T} 
The stack $ \mathcal { QG }^{\epsilon} _ { 0 , n } ( V // \mathbf { G } , \beta ) $ is a Deligne-Mumford stack, separated, of finite type, and the $\nu$-relative obstruction theory constructed above is perfect for all $ \epsilon $ and $ \mu $-relative obstruction theory is perfect when $ \epsilon=0^+ $. 
\end{theorem}

\subsection{Fixed loci of $\mathbb{C}^*$-action on graph spaces}
There is a natural $\mathbb{C}^*$ action on $ \mathcal { QG }^{\epsilon} _ { 0 , n } ( V // \mathbf { G } , \beta ) $, i.e. the $\mathbb{C}^*$ action on the distinguish component $\mathbb{P}^1$ by $t \left[ x _ { 0 } , x _ { 1 } \right] = \left[ t x _ { 0 } , x _ { 1 } \right] , \forall t \in \mathbb { C } ^ { * }$, let $0=[1,0]$, denoted by $q$ the weight of cotangent bundle over $0$. So the $ \mathbb{C}^* $ fixed loci should be the $\epsilon$-stable parametrized quasimaps that all the base points, markings, nodes and degree $ \epsilon $ should support on $0$ and $ \infty $, and the section $u$ send $ \mathbb{P}^1 $ to a point in $ V//\mathbf{G} $, i.e.
\begin{align*}
\left( \mathcal{Q G} _ { 0 , n , \beta } ^ { \epsilon } ( V / / \mathbf { G } ) \right) ^ { \mathbb { C } ^ { * } } = \bigsqcup F _ { 0 , n _ { 2 } , \beta _ { 2 } } ^ { 0 , n _ { 1 } , \beta _ { 1 } }
\end{align*}
the union over all possible splittings
\begin{align*}
n = n _ { 1 } + n _ { 2 } , \quad \beta = \beta _ { 1 } + \beta _ { 2 }
\end{align*}
where 
\begin{align*}
F _ { 0 , n _ { 2 } , \beta _ { 2 } } ^ { 0 , n _ { 1 } , \beta _ { 1 } } \cong  \mathcal{Q} _ { 0 , n_1+{\bullet} , \beta_{1} } ^ { \epsilon } ( V / / \mathbf { G } )\times_{V//\mathbf{G}} \mathcal{Q} _ { 0 , n_2+{\bullet} , \beta_{2} } ^ { \epsilon } ( V / / \mathbf { G } )
\end{align*}
is the fiber product over the evaluation map $ ev_{\bullet} $ at the special point $ \bullet $ when the moduli spaces are meaningful. Under this identification, the inclusion 
\begin{align*}
i: F _ { 0 , n _ { 2 } , \beta _ { 2 } } ^ { 0 , n _ { 1 } , \beta _ { 1 } } \rightarrow \mathcal{Q G} _ { 0 , n , \beta } ^ { \epsilon } ( V / / \mathbf { G } )
\end{align*}
could be described as follows, we glue two $ \epsilon $-stable quasimaps at $0$ and $ \infty $ to the constant map $ \mathbb{P}^1 \rightarrow p=ev_{\bullet}(\bullet) \in V//\mathbf{G} $.

Since $ \mathcal{Q G} _ { 0,1,0 } \left( V //( \mathbf { G } ) \cong V / / \mathbf { G } \times \mathbb { P } ^ { 1 }\right) $ with fixed point loci $V//\mathbf{G} \times \{0 \}$ and $V//\mathbf{G} \times \{ \infty \}$ and $ \operatorname{Q G} _ { 0,0,0 } ( V / / \mathbf { G } ) \cong V / / \mathbf { G } $ with trivial $\mathbb{C}^*$ action, then the unstable conditions $(0,n_1,\beta_1) = (0,1,0)$ or $(0,n_1,\beta_1) = (0,0,0)$ (similar for $(0,n_2,\beta_2)$) when $ \epsilon = \infty $ is included by
\begin{align*}
\mathcal{Q} _ { 0,0 + \bullet } ^ { \epsilon } ( V / / \mathbf { G } , 0 ) : = V / / \mathbf { G },\ \ \mathcal{Q} _ { 0,1 + \bullet } ^ { \epsilon } ( V / / \mathbf { G } , 0 ) : = V / / \mathbf { G }, \ \ ev_{\bullet} := id_{V//\mathbf{G}}
\end{align*}
However, there are more unstable conditions when $ \epsilon=0^+ $, since the stability conditions require the rational component having at least two marked points in this case, so condition $  (0,n_1,\beta_1)=(0,0,\beta_1) $ and $ \epsilon \beta_1 \leq 1 $ is unstable, we denote by
\begin{align*}
\operatorname{F}_{0, \beta} := \mathcal{Q} _ { 0,0 + \bullet } ( V / / \mathbf { G } , \beta ) _ { 0 }
\end{align*}
the moduli space parametrizing the quasimaps of class $\beta$
\begin{align*}
\left( \mathbb { P } ^ { 1 } , P , u \right)
\end{align*}
with $P$ a principal $\mathbf{G}$-bundle on $\mathbb{P}^1$, $u : \mathbb{P}^1 \rightarrow P \times_{\mathbf{G}} V $ a section such that $u(x) \in V^s$ for $x \neq 0 \in \mathbb{P}^1$ and $ 0 \in \mathbb{P}^1$ is a base point of length $\beta(L_\theta)$. Similarly, we have $ \operatorname{Q} _ { 0,0 + \bullet } ( V / / \mathbf { G } , \beta ) _ { \infty }$. And we define the 
\begin{align*}
e v _ { \bullet } \left( \left( \mathbb { P } ^ { 1 } , P , u \right) \right) = u _ { r e g } \left( \mathbb { P } ^ { 1 } \right)
\end{align*}
where the $ u_{reg} $ is the constant map from $ \mathbb{P}^1 \setminus \{0 \} $ to $ V//\mathbf{G} $. Then
we conclude

When $ k \geq 1 $ and $ \epsilon \leq \frac{1}{\beta_{1}(L_{\theta})} $,
\begin{align*}
F _ { 0 , n , \beta _ { 2 } } ^ { 0,0 , \beta _ { 1 } } \cong \mathcal { Q } _ { 0,0 + \bullet }  \left( V / / \mathbf { G } , \beta _ { 1 } \right) _ { 0 } \times _ { V // \mathbf { G } } \mathcal { Q } _ { 0 , n + \bullet } ^ { \epsilon } \left( V / / \mathbf { G } , \beta _ { 2 } \right)
\end{align*}

When $ k \geq 1 $ and $ \epsilon \leq \frac{1}{\beta_{2}(L_{\theta})} $,
\begin{align*}
F ^ { 0 , n , \beta _ { 1 } } _ { 0,0 , \beta _ { 2 } } \cong \mathcal { Q } _ { 0 , n + \bullet } ^ { \epsilon } \left( V / / \mathbf { G } , \beta _ { 1 } \right)
  \times _ { V // \mathbf { G } } \mathcal { Q } _ { 0,0 + \bullet }  \left( V / / \mathbf { G } , \beta _ { 2 } \right) _ { \infty }
\end{align*}

When $ k=g=0 $ and $ \epsilon \leq \min \left\{ \frac { 1 } { \beta _ { 1 } \left( L _ { \theta } \right) } , \frac { 1 } { \beta _ { 2 } \left( L _ { \theta } \right) } \right\} $,
\begin{align*}
F _ { 0 , 0 , \beta _ { 2 } } ^ { 0,0 , \beta _ { 1 } } \cong \mathcal { Q } _ { 0,0 + \bullet }  \left( V / / \mathbf { G } , \beta _ { 1 } \right) _ { 0 }  \times _ { V // \mathbf { G } } \mathcal { Q } _ { 0,0 + \bullet }  \left( V / / \mathbf { G } , \beta _ { 2 } \right) _ { \infty }
\end{align*}

\subsection{Level structure and $ \mathcal{J}^{R,l,\epsilon} $-function}
Here we follow the notation of \cite{2018arXiv180406552R}, first, recall the definition of the level structure and then use it to define the permutation-equivariant quasi-map K-theory invariants with level structure.

Let $\mathcal{X}$ be a Deligne-Mumford stack, and $ \mathcal{F}^{\bullet} $ be a complex of coherent sheaves on $\mathcal{X}$ which has a bounded locally free resolution, i.e., there exists a bounded complex of locally free, finitely generated $ \mathcal{O}_{\mathcal{X}} $ modules $ \mathcal{G}^{\bullet} $ and a quasi-isomorphism
\begin{align*}
\mathcal{G}^{\bullet} \rightarrow \mathcal{F}^{\bullet}
\end{align*}
the determinant line bundle associated to $ \mathcal{F}^{\bullet} $ is 
\begin{align*}
\operatorname{det} := \otimes_n \operatorname{det}(\mathcal{G}^n)^{(-1)^n}
\end{align*}
where $ \operatorname { det } ( \mathcal { E } ) : = \bigwedge ^ { \operatorname { rank } ( \mathcal { E } ) } \mathcal { E } $ for locally free sheaf $ \mathcal{E} $.
\begin{definition}
Given a finite dimensional representation $R$ of $\mathbf{G}$, we define the level $l$ determinant line bundle over  $ \operatorname { Q }^{ \epsilon } _ { g , n } ( V / / \mathbf { G } , \beta ) $ as
\begin{align*}
\mathcal { D } ^ { R , l } : = \operatorname { det } ^ { - l } R^{\bullet} \pi _ { * } \left( \mathfrak { P } \times _ { \mathbf{G} } R \right)
\end{align*}
here $ \pi : \mathcal { C } ^ { \epsilon } \rightarrow \mathrm { Q } _ { g , n } ^ { \epsilon } ( V / / \mathbf { G } , \beta )  $ and $ \operatorname { det } ^ { - l } ( \cdot ) : = \operatorname { det } ( \cdot ) ^ { - l } $ denotes the l-th power of the inverse of the determinant line bundle. Similarly, we can define level $ l $ determinant line bundle over  $ \mathcal { QG }^{ \epsilon } _ { g , n } ( V / / \mathbf { G } , \beta ) $, also denoted by $ \mathcal { D } ^ { R , l } $.
\end{definition}

Let $\Lambda$ be a $\lambda$-algebra, an algebra over $\mathbb{Q}$ with abstract Adams operator $\Psi^m : \Lambda \rightarrow \Lambda $, $m=1,2,\cdots$, such that, $\Psi ^ { r } \Psi ^ { s } = \Psi ^ { r s }$ and $\Psi^1=id$, we will assume $\Lambda$ containing Novikov variables. Define the loop space 
\begin{align*}
\mathcal { K } : = \left[ K ^ { 0 } ( V//\mathbf{G}) \otimes \Lambda \right] \otimes \mathbb { C } ( q ) 
\end{align*} 
there is a symplectic form $\Omega $ on $\mathcal{K}$ given by
\begin{align*}
\Omega ( f , g ) : = \left[ \operatorname { Res } _ { q = 0 } + \operatorname { Res } _ { q = \infty } \right] \left( f ( q ) , g \left( q ^ { - 1 } \right) \right) ^ { R , l } \frac { d q } { q }
\end{align*}
where $f , q \in \mathcal { K }$, and the pairing $( \ \ , \ \ )^{R,l}$ is the twisted pairing [7]. With respect to $\Omega$, there is a Lagrangian polarization  
\begin{align*}
\mathcal { K } = \mathcal { K } _ { + } \oplus \mathcal { K } _ { - }
\end{align*}
where
\begin{align*}
& \mathcal { K } _ { + } = \left[ K^0 ( V / / \mathbf{G} ) \otimes \Lambda \right] \otimes \mathbb { C } \left[ q , q ^ { - 1 } \right]  \\
& \mathcal { K } _ { - } = \{ f \in \mathcal { K } \ \ |\ \ f ( 0 ) \neq \infty , f ( \infty ) = 0 \}
\end{align*}

By Theorem 3.2, the $ \nu $-rellative obstruction theory is perfect, then from \cite{lee2004quantum} we have a virtual structure sheaf, denoted by $ \mathcal { O }^ { \mathrm { vir } } _ { \mathcal{Q} _ { g , n } ^ { \epsilon } ( W / / \mathbf{G} , \beta )} $. Consider a natural $ S_n $ action on $ \mathcal { Q }^{ \epsilon } _ { g , n } ( V / / \mathbf { G } , \beta ) $ by permuting the n marked points, the virtual structure sheaf and the determinant line bundle are invariant under this action. Then we have a $ S_n $-module
\begin{align*}
[ \mathbf { t } ( L ) , \ldots , \mathbf { t } ( L ) ] _ { g , k , \beta } : = \sum _ { m } ( - 1 ) ^ { m } {H} ^ { m } \left( \mathcal { O }^ { \mathrm { vir } } _ { \mathcal { Q } _ { g , n } ^ { \epsilon }   ( V / / \mathbf{G} , \beta )} \otimes \mathcal { D } ^ { R , l } \otimes _ { i = 1 } ^ { k } \mathbf { t } \left( L _ { i } \right) \right)
\end{align*}
where $ \mathbf { t } ( q ) \in \mathcal { K } _ { + } $. 

The correlators of the permutation-equivariant quasimap K-theory of level
$l$ are defined as
\begin{align*}
\langle \mathbf { t } ( L ) , \ldots , \mathbf { t } ( L ) \rangle _ { g , n , \beta } ^ { R , l , \epsilon , S _ { n } } : = pt _ { * } \left( \mathcal { O } _ { \mathcal{Q} _ { g , n } ^ { \epsilon } ( V // \mathbf{G}, \beta ) } ^ { \mathrm { vir } } \otimes \mathcal { D } ^ { R , l } \otimes _ { i = 1 } ^ { k } \mathbf { t } \left( L _ { i } \right) \right)
\end{align*}
where $ pt_* $ is the K-theoretic pushforward along the projection
\begin{align*}
pt : \left[ \mathcal{ Q } _ { g , n } ^ { \epsilon } ( V / / \mathbf{G} , \beta ) / S _ { n } \right] \rightarrow [ \mathrm { pt } ]
\end{align*}
by definition, extracting $S_n$-invariants from the $S_n$-module defined above.

\begin{definition}
The permutation-equivariant K-theoretic $\mathcal{J}^{R,l,\epsilon}$-function of $V // \mathbf{G}$ of level $l$ is defined as
\begin{align*}
\mathcal { J } _ { S _ { \infty } } ^ { R , l , \epsilon } ( \mathbf { t } ( q ) , Q ) 
 &:= \sum_{k \geq 0, \beta \in {\operatorname { Eff } ( V , \mathbf { G } , \theta )} } Q^{\beta} (ev_{\bullet})_{*} [\operatorname{Res}_{\operatorname{F}_{0,\beta}}( \mathcal{QG} _ { 0 , n } ^ { \epsilon } ( V / / \mathbf { G } , \beta )_{0})^{\mathrm{vir}} \otimes \mathcal { D } ^ { R , l } \otimes _ { i = 1 } ^ { n } \mathbf { t } ( L _ { i } ) ]^{S_n} \\
 &:= 1 + \frac { \mathbf { t } ( q ) } { 1 - q } +\sum _ { a } \sum _ { \beta \neq 0 } Q ^ { \beta } \chi \left( \operatorname{F} _ { 0 , \beta } , \mathcal { O } _ { \operatorname{F} _ { 0 , \beta } } ^ { \mathrm { vir } } \otimes { ev }_{\bullet} ^ { * } ( \phi _ { a } ) \otimes \left( \frac { \operatorname { t r } _ { \mathbb { C } ^ { * } } \mathcal { D } ^ { R , l } } { \lambda_{-1}^{\mathbb{C}^*}  N _ { \operatorname{F} _ { 0 , \beta } } ^ { \vee } } \right) \right) \phi ^ { a } \\
 &+ \sum_a \sum _ { n \geq 1 or \beta(L_\theta) \geq \frac{1}{\epsilon} \atop ( n , \beta ) \neq ( 1,0 ) } Q ^ { \beta } \left\langle \frac { \phi _ { a } } { ( 1 - q ) ( 1 - q L ) } , \mathbf { t } ( L ) , \ldots , \mathbf { t } ( L ) \right\rangle _ { 0 , n + 1 , \beta } ^ { R , l , \epsilon , S _ { n } } \phi ^ { a }
\end{align*}
where $\{ \phi_\alpha \}$ is a basis of $K^0(V//\mathbf{G})$ and $\{ \phi^\alpha \}$ is the dual basis with respect to twisted pairing  $( \ \ , \ \ )^{R,l}$ .
\end{definition}

\begin{definition}
The $\mathcal{J}^{R,l,\infty}$-function is a function 
\begin{align*}
(1-q)\mathcal{J}^{R,l,\infty} : \mathcal{K}_{+} \rightarrow \mathcal{K}
\end{align*}
we call $\mathcal { L } _ { S _ { \infty } }$ the range of $\mathcal{J}^{R,l,\infty}$-function, defined by
\begin{align*}
\mathcal { L } _ { S _ { \infty } } : = \bigcup _ { \mathfrak { t } ( q ) \in \mathcal { K } _ { + } } ( 1 - q ) \mathcal { J } _ { S _ { \infty } } ^ { R , l , \infty } ( \mathbf { t } ( q ) , Q ) \subset \mathcal { K }
\end{align*} 
\end{definition}

\begin{definition}
When taking $\epsilon$ small enough, denoted by $\epsilon=0^{+}$, we call $\mathcal{J}^{R,l,0^{+}}(0)$ the small $I$-function of level $l$, i.e,
\begin{align*}
{I}^{R,l}(q,Q):= \mathcal { J } _ { S _ { \infty } } ^ { R , l , 0^{+} } ( 0 , Q ) =  1 + \sum _ { \beta \geq 0 } Q ^ { \beta } (ev_{\bullet})_{*} \left(  \mathcal { O } _ { \operatorname{F} _ { 0 , \beta } } ^ { \mathrm { vir } } \otimes  \left( \frac { \operatorname { t r } _ { \mathbb { C } ^ { * } } \mathcal { D } ^ { R , l } } { \lambda_{-1}^{\mathbb{C}^*}  N _ { \operatorname{F} _ { 0 , \beta } } ^ { \vee } } \right) \right)  
\end{align*}
\end{definition}

\subsection{A Mirror theorem for partial flag variety}

Let $X$ be $F l _ { r _ { 1 } , \ldots , r _ { I } } \left( \mathbb { C } ^ { n } \right)= \{0 \subset \mathbb{C}^{r_1} \subset \cdots \subset \mathbb{C}^{r_I} \subset \mathbb{C}^n \}$ the manifold of partial flags in $\mathbb{C}^n$. It admits the an embedding into the product of Grassmannian 
\begin{align*}
\psi: X \hookrightarrow Y :=\operatorname { Gr } \left( r _ { 1 } , \mathbb{C}^n \right) \times \cdots \times G r \left( r _ { I } , \mathbb{C}^n \right) 
\end{align*}
together with Pl\"ucker embedding, we have 
\begin{align*}
j: Y \hookrightarrow \Pi := \prod_{i=1}^{I} \mathbb{P}^{n_i-1}, n_{i} :=\left( \begin{array}{c}{n} \\ {r_i}\end{array}\right)
\end{align*}
Let $(x : y)$ be homogeneous coordinates on $\mathbb{P}^1$. A degree $d$ holomorphic map $\mathbb{P}^1 \rightarrow \mathbb{P}^N$ is uniquely determined, up to a constant scalar factor, by $N + 1$ relatively prime degree d binary forms $ (f_0(x : y) : \ldots : f_N(x : y))$. Omitting the condition that the forms are relatively prime we compactify the space of degree
$d$ holomorphic maps  $\mathbb{P}^1 \rightarrow \mathbb{P}^N$ to a complex projective space of dimension $(N + 1)(d + 1) - 1$, denoted by $\mathbb{P}^N_d$. Similarly, we have $\Pi_d$. 

As we see, $\mathcal{QG}_{0,0}(F l _ { r _ { 1 } , \ldots , r _ { I } } \left( \mathbb { C } ^ { n } \right),d)$ is the hyper-quot scheme, parametrize sequence of bundles $E^{r_1} \rightarrow \ldots \rightarrow E^{r_I} \rightarrow E^{r_{I+1}} = \mathbb{C}^n$ on $\mathbb{P}^1$, the map $\nu$ is given by top exterior powers $\wedge^{i} E^{r_i} \rightarrow  \wedge^{i} \mathbb{C}^n$. A bi-degree $(1, d)$ rational curve in $\mathbb{P}^1 \times X$ projected to $\mathbb{P}^1 \times \mathbb{P}^{n_i-1}$ by the Pl\"ucker map consists of the graph $\Sigma_0$ of a degree $m_0 \leq d_i$ map $\mathbb{P}^1 \rightarrow \mathbb{P}^{n_i-1}$ and a few vertical curves $\Sigma_j$ of bi-degrees $(0,m_j)$ with $\sum m_j = d_i - m_0$, attached to the graph, see \cite{givental2003quantum} for details.  $\widebar{GM}_{0,0}(X,d)$ and $\mathcal{QG}_{0,0}(X,d)$ are rational desingularization of Drinfiled's compactification of the space of degree d maps from $\mathbb{P}^1$ to $X$, i.e. the closure of $\operatorname{Map}_d(\mathbb{P}^1,X)$ in $\Pi_d$\cite{givental2003quantum}. So we have following commutative diagram:
\begin{align}
\xymatrix{ \widebar{GM}_{0,0}(X,d) \ar[rr]^{\mu} & & \Pi _d  & & \ar[ll]^{\nu} \mathcal{QG}_{0,0}(X,d)  \\  & & \Pi\ar[u]^{\alpha_d}  &  & \\ \widebar{M}_{0,1}(X,d) \ar[uu]^{\alpha _d^X}  \ar[rr] \ar[rrd]^{ev}  &  & Y \ar[u]^{j}  & & F_0 \ar[ll]^{} \ar[uu]^{\alpha _F} \ar[llu]^{f} \ar[lld]^{\tilde{ev}  } \\ &  &  X \ar[u]^{\psi} &  &  }
\end{align}

By definition, the small $J$-function is defined by 
\begin{align*}
J_d := ev_* \left(  \frac{\mathcal{O}_{\widebar{M}_{0,1}(X,d)}}{\lambda_{-1}(N^{\vee}_{\widebar{M}_{0,1}(X,d)/\widebar{GM}_{0,0}(X,d)})} \right) = ev_* \left(  \frac{\mathcal{O}_{\widebar{M}_{0,1}(X,d)}}{(1-q)(1-qL)} \right)
\end{align*}
and small $I$-function is defined by
\begin{align*}
I_d := \tilde{ev}_* \left(  \frac{\mathcal{O}_{F_0}}{\lambda_{-1}(N^{\vee}_{F_0/\mathcal{QG}_{0,0}(X,d)})}  \right)
\end{align*}

\begin{lemma}{[Correspondence of residues \cite{taipale2013k}]} Let $X$, $Y$ be nonsingular schemes or smooth Deligne-Mumford stacks with a $T=(\mathbb{C}^*)^n$-action and let $g : X \rightarrow Y$ be a proper equivariant mrophism and $W= \{W_k \}$, $V$ be components in the torus-fixed loci of $X$ and $Y$ respectively, such that $g(W) \subset V $. We have
\begin{align*}
\xymatrix{ W \ar[rr]^{i} \ar[d]_{f}  & & X \ar[d]_{g} \\ 
V \ar[rr]^{j} & & Y }
\end{align*}
Then for a coherent sheaf $\mathcal{H}$ on $W$, let $\mathcal{F} = i_*(\mathcal{H})$, then 
\begin{align*}
\frac{j^* g_* (\mathcal{F})}{\lambda_{-1}(N^{\vee}_{V/Y})} = \sum _i  f_* \frac{ i^*(\mathcal{F})}{\lambda_{-1}(N^{\vee}_{W_i/X})}
\end{align*}
\end{lemma}
\begin{proof}
\begin{align*}
j^*g_*(i_* \mathcal{H}) = j^*j_*f_* \mathcal{H} = f_* \mathcal{H} \cdot \lambda_{-1}(N^{\vee}_{V/Y})
\end{align*}
and 
\begin{align*}
i^*i_* \mathcal{H} = \mathcal{H} \cdot \lambda_{-1}(N^{\vee}_{W/X})
\end{align*}
so we have
\begin{align*}
\frac{j^* g_* (\mathcal{F})}{\lambda_{-1}(N^{\vee}_{V/Y})} = \sum _i  f_* \frac{ i^*(\mathcal{F})}{\lambda_{-1}(N^{\vee}_{W_i/X})}
\end{align*}
\end{proof}

Applying $i := j \circ \psi $ to $J$-function, we get
\begin{align*}
i_* J &= i_* ev_* \left(  \frac{\mathcal{O}_{\bar{M}_{0,1}(X,d)}}{(1-q)(1-qL)} \right) = i_* ev_* \left(  \frac{(\alpha^X_d)^*\mathcal{O}_{\bar{GM}_{0,0}(X,d)}}{1-qL} \right) \\
      &= \frac{\alpha^*_d \mu_* \mathcal{O}_{\bar{GM}_{0,0}(X,d)}}{\lambda_{-1}(N^{\vee}_{\mathcal{P}^N/\mathcal{P}^N_d})} = \frac{\alpha^*_d \mathcal{O}_{\mathcal{P}^N_d}}{\lambda_{-1} (N^{\vee}_{\mathcal{P}^N/\mathcal{P}^N_d})}
\end{align*}
where the third equation comes from correspondence of residue and the fourth equation comes from the rational desingulariztion.
 
On the other hand, similarly,
\begin{align*}
i_* I : & = i_* \tilde{ev} _* \left( \frac{\alpha ^*_F \mathcal{O}_{\mathcal{QG}_{0,0}(X,d) }}{\lambda_{-1}(N^{\vee}_{F_0/\mathcal{QG}_{0,0}(X,d)})}  \right) = \frac{\alpha^*_d \nu^* \mathcal{O}_{\mathcal{QG}_{0,0}(X,d)}}{\lambda_{-1}(N^{\vee}_{\mathbb{P}^N/\mathbb{P}^N_d})} \\
        & = \frac{\alpha^*_d \mathcal{O}_{\mathbb{P}^N_d}}{\lambda_{-1}(N^{\vee}_{\mathbb{P}^N/\mathbb{P}^N_d})}
\end{align*}

\begin{lemma}
If we consider a big torus $\mathbb{T}=(\mathbb{C}^*)^n$-action on $\mathbb{C}^n$, then the push-forward of equivariant K groups
\begin{align*}
i_* : K_{\mathbb{T}}^{0}(F l _ { r _ { 1 } , \ldots , r _ { I } } \left( \mathbb { C } ^ { n } \right)) \rightarrow K_{\mathbb{T}}^0(\Pi)
\end{align*}
is injective.
\end{lemma}
\begin{proof}
This comes from the localization formula and Lemma 4.1, indeed, 
\begin{align*}
K_{\mathbb{T}}^0(X) \cong K_{\mathbb{T}}^0(X^{\mathbb{T}})
\end{align*}
where $X$ stands for $F l _ { r _ { 1 } , \ldots , r _ { I } } \left( \mathbb { C } ^ { n } \right)$ or $\Pi$, and in both cases, the $\mathbb{T}$-fixed loci $X^{\mathbb{T}}$ are fixed points, then by injective map from fixed points to fixed points and correspondence of residue, we get the conclusion. 
\end{proof}

From the above discussion and Lemma 2.3, we arrive at the following theorem by taking the non-equivariant limit,

\begin{theorem}{\cite{taipale2013k}}
The small $I$-function equals to the small $J$-function.
\end{theorem}

\section{K theoretic abelian and non-abelian correspondence for I-function with level structures}
\subsection{Identify the fixed locus and evaluation map}
From Definition 2.11, we know it's important to figure out evaluation map $ev_{\bullet}$ and fixed loci $F_{0,\beta}$. Here we follow \cite{webb2018abelian}, let $ V $ is a vector space, $ \mathbf{G} $ is a connected reductive group, and $ \mathbf{T} \subset \mathbf{G} $ is the maximal torus. Let $\rho : \mathbf{G} \rightarrow GL(V) $ be the representation. Let $\mathbf{S}$ be a torus in $GL(V)$ that commutes with $\rho(\mathbf{G})$. In the following, we consider GIT quotient $ V//_{\theta} \mathbf{G} $.
\begin{definition}
The $\mathbf{T}$-degree $W \tilde{\beta}$ of a quasimap $(\mathcal{P},u)$ is the $W$-orbit of the homomorphism $\tilde{\beta} \in \mathrm{Hom}(\chi(\mathbf{T}),\mathbb{Z})$ given by $\tilde{\beta}(\xi)=\rm{deg}_{\mathbb{P}^1}(\mathcal{J}\times_{\mathbf{T}}\mathbb{C}_{\xi})$, where $\mathcal{J}$ is a principal $\mathbf{T}$-bundle associated to $\mathcal{P}$.
\end{definition}
The $\mathbf{T}$-degree of a principal $\mathbf{G}$-bundle $\mathcal{P}$ has the following properties:

$\bullet$ The $\mathbf{T}$-degree of $\mathcal{P}$  determines $\mathcal{P}$  up to isomorphism.

$\bullet$ The natural map
$\tau : \operatorname { Hom } ( \chi ( \mathbf{T} ) , \mathbb { Z } ) \rightarrow \operatorname { Hom } ( \chi ( \mathbf{G} ) , \mathbb { Z } )$
sends a representative of the $\mathbf{T}$-degree of $\mathcal{P}$ to the degree of $\mathcal{P}$.

Since {$\mathbf{S}$ commutes with $\mathbf{G}$, then the group generated by $\mathbf{S}$ and $\rho(\mathbf{T})$ is a torus. Choose a basis of V that diagonalizes this torus, and denote the associated weights of the $\mathbf{T}$-action by $\xi_1,\cdots,\xi_n \in \chi(\mathbf{T})$. Let $\mathcal{P} \rightarrow \mathbb{P}^1$ be the principal $\mathbf{G}$-bundle of $\mathbf{T}$-degree $W \tilde{\beta}$ such that $\mathcal{P} \times_{\mathbf{G}} V=\oplus_{j=1}^{n} \mathcal{O}(\tilde{\beta}(\xi_j))$.

From \cite{webb2018abelian}, we can get the following description of fixed loci.

\begin{proposition}\cite{webb2018abelian}
For any $\tilde{\beta} \in \mathrm{Hom}(\chi(\mathbf{T} ), \mathbb{Z})$, there is \\
$\bullet$ a parabolic subgroup $P_{\tilde{\beta}} \subset \mathbf{G}$, \\
$\bullet$ a subspace $V_{\tilde{\beta}} \subset  V$ stable under $P_{\tilde{\beta}}$, and \\
$\bullet$ a morphism $\Xi_{\tilde{\beta}} : \left( V _ { \tilde { \beta } } \cap V ^ { s } ( \mathbf{G} ) \right) / P _ { \tilde { \beta } } \rightarrow F _ { \beta }$ that is an isomorphism onto a connected component of $F _ { \beta }$\\
The image of $\Xi_{\tilde{\beta}}$ depends only on the $W$-class of $ \tilde{\beta} $, and if we denote the image $ F _ { W \tilde { \beta } } $, then we have an $S$-equivariant commuting diagram
\begin{align*}
\xymatrix{ \left( V _ { \tilde { \beta } } \bigcap V ^ { s } ( \mathbf{G} ) \right) / \mathbf{P} _ { \tilde { \beta } } \ar[r]^-{\Xi_{\tilde{\beta}}}\ar[d]_{i} & F _ { W \tilde { \beta } } \ar[d]_{ev_\bullet} \\  V ^ { s } ( \mathbf{G} ) / \mathbf{P} _ { \tilde { \beta } } \ar[r]^{h} &  V ^ { s } ( \mathbf{G} ) / \mathbf{G} }
\end{align*}
\end{proposition}
more precisely, let $\Gamma_{\tilde{\beta}}$ be the subspace of $\Gamma(\mathcal{P} \times_{\mathbf{G}} V )$ of sections whose polynomial representations depend only on $x_0$, $V_{\tilde{\beta}}$ is the image of $\Gamma_{\tilde{\beta}}$ under the evaluation map $ev_{\bullet}$, let $\rm{Aut}_{\tilde{\beta}}$ be the image of $\operatorname { Aut } ( \mathcal { P } ) \rightarrow \operatorname { Aut } \left( \mathcal { P } \times _ \mathbf{ G } V \right)$, acts on $\Gamma_{\tilde{\beta}}$ by left multiplication, and $ev_{\bullet}$ identifies $\rm{Aut}_{\tilde{\beta}}$ with a parabolic subgroup of $\mathbf{G}$, denoted by $P_{\tilde{\beta}}$. Then the above proposition just says $\Gamma _ { \tilde { \beta } } ^ { s } / \mathrm { Aut } _ { \tilde { \beta } } \cong V _ { \tilde { \beta } } ^ { s } / P$.

\subsection{Push-forward formula of flag bundles}
We will need this push-forward formula of flag bundles in the following subsections. Let's first recall some properties about line bundles associated with characters \cite{tu2010computing}. Suppose the maximal torus $\mathbf{T}$ acts freely on the right on a topological space $X$ so that $X \rightarrow X/\mathbf{T}$ is a principal $\mathbf{T}$ -bundle. As in 2.1, for $\theta \in \chi(\mathbf{T})$, we can associate a line bundle on $X/\mathbf{T}$ by
\begin{align*}
L_{\theta} := X \times_{\mathbf{T}} \mathbb{C}_{\theta}
\end{align*}
The Weyl group of a maximal torus $\mathbf{T}$ in the compact, connected Lie group $\mathbf{G}$ is $W = N_\mathbf{G}(\mathbf{T} )/\mathbf{T}$, where $N_\mathbf{G}(\mathbf{T} )/\mathbf{T}$ is the normalizer of $\mathbf{T}$ in $\mathbf{G}$. The Weyl group is a finite reflection group. And the Weyl group $W$ acts on the character of $\mathbf{T}$ by
\begin{align*}
( w \cdot \theta ) ( t ) = \theta \left( w ^ { - 1 } t w \right)
\end{align*}
acts on the right on $X/\mathbf{T}$ by
\begin{align*}
r _ { w } ( x \mathbf{T} ) = ( x \mathbf{T} ) w = x w \mathbf{T} 
\end{align*}
Then we have
\begin{proposition}\cite{tu2010computing} 
The action of the Weyl group $W$ on the associated line bundles over $X/\mathbf{T}$ is compatible with its action on the characters of $\mathbf{T}$; more precisely, for $w \in W$ and $\theta \in  \chi(\mathbf{T})$,
\begin{align*}
w \cdot L_{\theta} = r _ { w } ^ { * } L_{\theta} \cong L_{w \cdot \theta}
\end{align*}
\end{proposition}
Consider the action of $\mathbf{T}$ on $\mathbf{G}/\mathbf{H}$ by left multiplication, where $\mathbf{H}$ is a closed subgroup of $\mathbf{G}$ containing $\mathbf{T}$, we have 
\begin{proposition}\cite{tu2010computing}
The fixed point set $F$ of the action is $W_{\mathbf{G}}/W_{\mathbf{H}} = N_{\mathbf{G}}(\mathbf{T})/N_{\mathbf{H}}(\mathbf{T})$.
\end{proposition}
and from the following lemma, we know the restriction of a character line bundle to the fixed point $w \in \mathbf{G}/\mathbf{H}$
\begin{lemma}\cite{tu2010computing}
At the fixed point $w = x\mathbf{H} \in W_{\mathbf{G}}/W_{\mathbf{H}}$, the torus $\mathbf{T}$ acts on the fiber of the line bundle $L_{\theta}$ as the representation $w \cdot \theta $, i.e., $(L_{\theta} )_{w} = \mathbb{C}_{w \cdot \theta}$.
\end{lemma}
Suppose the adjoint representation of $\mathbf{T}$ on $\mathfrak{g}$ decomposes $\mathfrak{g}$ into a direct sum
\begin{align*}
\mathfrak { g } = \mathfrak { t } \oplus \left( \bigoplus _ { \alpha \in R ^ { + } } \mathbb { C } _ { \alpha } \right)
\end{align*}
where $ R ^ { + }$ is a choice of positive roots, so by the above proposition the tangent bundle of $\mathbf{G}/\mathbf{T}$ is 
\begin{align*}
T ( \mathbf{G}/\mathbf{T} ) \simeq \mathbf{G} \times _ { \mathbf{T} } ( \mathfrak { g } / \mathfrak { t } ) \simeq \mathbf{G} \times _ { \mathbf{T} } \left( \bigoplus _ { \alpha \in R ^ + } \mathbb { C } _ { \alpha } \right) \simeq \bigoplus _ { \alpha \in R ^ { + } } L _ { \alpha }
\end{align*}
Since $\mathbf{T} \subset \mathbf{H}$, the representation of $\mathbf{H}$ can restrict to a representation of $ \mathbf{T} $, and let $ \sigma : \mathbf{G}/\mathbf{T} \rightarrow \mathbf{G}/\mathbf{H} $ be the projection, then
\begin{proposition}\cite{tu2010computing}
Under $\sigma$ the associated bundle $\mathbf{G} \times_{\mathbf{H}} V$ pulls back to $\mathbf{G} \times_{\mathbf{T}} V$:
\begin{align*}
\sigma ^ { * } \left( \mathbf{G} \times _ { \mathbf{H} } V \right) \simeq \mathbf{G} \times _ { \mathbf{T} } V
\end{align*}
\end{proposition}
\begin{remark}
The same for pullback $\sigma: X/\mathbf{T} \rightarrow X/\mathbf{H}$.
\end{remark}
So by the Proposition 3.5 and 3.6 we can get the pulling back the tangent bundle of $\mathbf{G}/\mathbf{H}$ to $\mathbf{G}/\mathbf{T}$. 
\begin{proposition}\cite{tu2010computing}
Under the natural projection $ \sigma : \mathbf{G}/\mathbf{T} \rightarrow \mathbf{G}/\mathbf{H} $, the tangent bundle $T(\mathbf{G}/\mathbf{H})$ pulls back to a sum of associated line bundles:
\begin{align*}
\sigma ^ { * } T ( \mathbf{G}/\mathbf{H} ) \simeq \bigoplus _ { \alpha \in R^{+} \backslash R ^ { + } ( \mathbf{H} ) } L _ { \alpha }
\end{align*}
\end{proposition}

If we take $\mathbf{H}$ to be Borel subgroup $\mathbf{B}$ in $\mathbf{G}$, since $ \mathbf{T} \cong \mathbf{B}/[\mathbf{B},\mathbf{B}] $, then the character group of $\mathbf{B}$, denoted by $\chi(\mathbf{B})$, is isomorphic to the character group $\chi(\mathbf{T})$ of $\mathbf{T}$, and since $\mathbf{B}$ is solvable, by Lie-Kolchin theorem all irreducible representations of $\mathbf{B}$ are one-dimensional, therefore, the representation rings of $\mathbf{B}$ and $\mathbf{T}$ are the same, i.e. $R(\mathbf{B}) \cong R(\mathbf{T})$. 

As above, we denote by $R$ the root system of $(\mathbf{G},\mathbf{T})$ and by $R^+$ the set of positive roots in an unusual choice of positive roots for $(\mathbf{G},\mathbf{T})$ by declaring the weights of the adjoint $\mathbf{T}$-action on Lie algebra $\mathfrak{b}$ of $\mathbf{B}$ to be the negative roots, this unusual choice is called the geometric choice, see 6.1.3\cite{chriss2009representation} for details. 

Let $\mathbf{P} \supset \mathbf{B} $ be a parabolic subgroup of $\mathbf{G}$ and let $L$ be the Levi subgroup of $\mathbf{P}$ containing $\mathbf{T}$, with root system $R_L$ and Weyl group $W_L$. We have the push forward and pullback formula from the following diagram

\begin{center}
\begin{tikzcd} 
 & V^s(\mathbf{G})/\mathbf{T} \ar{ldd}[swap]{\bar{\psi}} \ar{rdd}{\psi} &\\
 & & & \\
V^s(\mathbf{G})/\mathbf{P} \ar{rr}{h}
&
& V^s(\mathbf{G})/\mathbf{G}
\end{tikzcd}
\end{center}

\begin{lemma}
For any $V \in R(\mathbf{P})$, generated by positive roots, we have
\begin{align*}
\psi ^ { * } h _ { * } \left( E _ { V } \right) = \sum _ { w \in W / W _ { L } } w \left[ \frac { \bar{\psi} ^ { * } \left( E _ { V } \right) } { \prod _ { \alpha \in R ^ { + } \backslash R _ { L } } \left( 1 - L _ { \alpha } ^ { \vee } \right) } \right]
\end{align*}
\end{lemma}
\begin{proof}
We will state the proof from [XX]. First, notice that $\psi ^ { * } h _ { * } \left( E _ { V } \right)$ is an associated $\mathbf{T}$-module $H^0(\mathbf{G}/\mathbf{P}, \mathbf{G} \times_{\mathbf{P}} V)$. Next, consider left $\mathbf{T}$-action on $\mathbf{G}/\mathbf{T}$, by Prop 3.2, we know the fixed points are characterized by $W/W_L=W_{\mathbf{G}}/W_{\mathbf{P}}$, then together Prop 3.3 and 3.5, the localization formula says
\begin{align*}
\mathrm{Tr_t}( H^0(\mathbf{G} \times_{\mathbf{P}} V) )
&= \sum _ { w \in W / W _ { L } } \frac { i _ { w } ^ { * } \left( \mathbf{G} \times _ \mathbf{ P } V \right) } { \lambda _ { - 1 } \left( T _ { p _ { w } } ^ { \vee } \right) } \\
&= \sum _ { w \in W / W _ { L } } w \left[ \frac { V } { \prod _ { \alpha \in R ^ { + } \left \backslash R _ { L } \right. } \left( 1 -  L^{\vee}_\alpha \right) } \right]
\end{align*} 
where $i_w: p_w \rightarrow \mathbf{G}/\mathbf{P}$ is the embedding, and $\mathrm{Tr_t}$ denotes the character of $\mathbf{T}$. Finally, by considering vector bundles on $V^{s}(\mathbf{G})/\mathbf{T}$ associated with the $\mathbf{T}$-modules on both sides of above equation, we obtain the push-forward and pull back formula. 
\end{proof}

\subsection{Abelian and non-abelian correspondence for I-function with level structures}
Recall small $\mathrm{I} $-function, $I^{R,l}(q,Q)$,
\begin{align*}
 {I}^{R,l}(q,Q) = 1 + \sum _ { \beta \neq 0 } Q ^ { \beta } (ev_{\bullet})_{*} \left(  \mathcal { O } _ { \operatorname{F} _ { 0 , \beta } } ^ { \mathrm { vir } } \otimes  \left( \frac { \operatorname { t r } _ { \mathbb { C } ^ { * } } \mathcal { D } ^ { R , l } } { \lambda_{-1}^{\mathbb{C}^*}  N _ { \operatorname{F} _ { 0 , \beta } } ^ { \vee } } \right) \right) 
\end{align*}
and
\begin{align*}
\operatorname{F}_{0, \beta} := \operatorname{Q} _ { 0,0 + \bullet } ( V / / \mathbf { G } , \beta ) _ { 0 }
\end{align*}
the moduli space parametrizing the quasimaps of class $\beta$
\begin{align*}
\left( \mathbb { P } ^ { 1 } , \mathcal{P} , u \right)
\end{align*}
where $ u : \mathbb{P}^1 \rightarrow \mathcal{P} \times_{\mathbf{G}} V $ a section such that $u(x) \in V^s$ for $x \neq 0 \in \mathbb{P}^1$ and $ 0 \in \mathbb{P}^1$ is a base point of length $\beta(L_\theta)$.

The $\mu$-relative obstruction theory on fixed loci $ \mathrm{F}_{0,\beta} $ by definition is
$$ \left( R ^ { \bullet } \pi _ { * } \mathcal { F } \right) ^ { \vee } |_{\mathrm{F}_{A}} $$
of the following exact sequence
\begin{align}
0 \longrightarrow \mathfrak { P } \times _ { \mathbf { G } } \mathfrak { g } \longrightarrow \mathfrak { P } \times _ { \mathbf { G } } V \longrightarrow \mathcal { F } \longrightarrow 0
\end{align}
so,
\begin{align*}
\psi^* {I}_{\beta}^{V//\mathbf{G},R,l}(q,Q) 
&= \sum_{\tilde{\beta}} \psi^*(ev_{\bullet})_{*} \left(  \mathcal { O } _ { \operatorname{F} _ { {\beta} } } ^ { \mathrm { vir } } \otimes  \left( \frac { \operatorname { t r } _ { \mathbb { C } ^ { * } } \mathcal { D } ^ { R , l } } { \lambda_{-1}^{\mathbb{C}^*}  N _ { \operatorname{F} _ { {\beta} } } ^ { \vee } } \right) \right)  \\
&= \sum_{W\tilde{\beta}} \psi^*h_{*}i_{*} \left(  \mathcal { O } _ { \operatorname{F} _ { W\tilde{\beta} } } ^ { \mathrm { vir } } \otimes  \left( \frac { \operatorname { t r } _ { \mathbb { C } ^ { * } } \mathcal { D } ^ { R , l } } { \lambda_{-1}^{\mathbb{C}^*}  N _ { \operatorname{F} _ { W\tilde{\beta} } } ^ { \vee } } \right) \right)
\end{align*} 
then by the following lemma and lemma 3.10, we get,
\begin{align}
\nonumber & \psi^*{I}_{\beta}^{V//\mathbf{G},R,l}(q,Q) = \sum_{W\tilde{\beta} \rightarrow \beta} \sum_{w \in W/W_L} w  [  \frac { \lambda^{\mathbf{S}} _ { -1 } ( V ^ { s } \times _ \mathbf{ T } V )^{\vee} } { \lambda^{\mathbf{S} } _ { -1 } ( V ^ { s } \times _ \mathbf{ T } V _ { \tilde { \beta } } )^{\vee} }  \cdot   \\
& \frac{ \lambda^{\mathbb{C}^*} _ { -1 }(\bar{\psi}^* \mathrm{R}^1\pi_*(\mathfrak{P} \times_{\mathbf{P}} V )^{mov})^{\vee} \cdot  \lambda^{\mathbb{C}^*} _ { -1 }(\bar{\psi}^* \mathrm{R}^0\pi_*( \mathfrak{P} \times_{\mathbf{P}} \mathfrak{g})^{mov})^{\vee}} {\lambda^{\mathbb{C}^*} _ { -1 }(\bar{\psi}^* \mathrm{R}^1\pi_*(\mathfrak{P} \times_{\mathbf{P}} \mathfrak{g})^{mov})^{\vee} \cdot \lambda^{\mathbb{C}^*} _ { -1 }(\bar{\psi}^* \mathrm{R}^0\pi_*(\mathfrak{P} \times_{\mathbf{P}} V )^{mov})^{\vee}} 
\frac{\mathrm{det}^{-l} (\bar{\psi}^* \mathrm{R}^{\bullet}\pi_*(\mathfrak{P} \times_{\mathbf{P}} R))}{\prod _ { \alpha \in R ^ { + } \backslash R _ { L } } \left( 1 -  L^{\vee}_\alpha \right)} ] 
\end{align}
\begin{remark}
Here we should view the $\tilde{\beta}$ in the right hand side as a representative in W-orbit $W\tilde{\beta}$.
\end{remark}

\begin{lemma}
Let $i : F_{W{\tilde{\beta}}}  \rightarrow V ^ { s } / \mathbf{P}$ be the inclusion of smooth varieties given by evaluation at $(1,0)$. Let $\gamma$ be in $K^{\mathbf{S}}_0(V^s/\mathbf{P})$. Then we get
\begin{align*}
i ^ { * } i _ { * } \gamma = \frac { \lambda^{\mathbf{S}} _ { -1 } \left( V ^ { s } \times _ { P } V \right) ^{\vee}} { \lambda^{\mathbf{S}} _ { -1 } ( V ^ { s } \times _ { P } V _ { \tilde { \beta } } ) ^{\vee}} \gamma
\end{align*}
\end{lemma}
\begin{proof}
First recall the generalized "Euler sequence"
\begin{align*}
0 \longrightarrow V ^ { s } \times _ { \mathbf { G } } \mathfrak { g } \longrightarrow V ^ { s } \times _ { \mathbf { G } } V \rightarrow T _ { V // \mathbf { G } } \longrightarrow 0
\end{align*}
Then it follows from the projection formula and above Euler exact sequence.
\end{proof}
Since $\bar{\psi}$ is flat, then $\bar{\psi}^*$ commutes with $\mathrm{R}^{\bullet}\pi_*$, and by Prop 3.5, we get,
\begin{align*}
\bar{\psi}^*(\mathfrak{P}  \times_{\mathbf{P}} V) = \mathfrak{P}  \times_{\mathbf{T}} V,\ \ 
\bar{\psi}^*(\mathfrak{P}  \times_{\mathbf{P}} \mathfrak{g}) = \mathfrak{P}  \times_{\mathbf{T}} \mathfrak{g},\ \ 
\bar{\psi}^* (\mathfrak{P}  \times_{\mathbf{P}} R) = \mathfrak{P} \times_{\mathbf{T}} R
\end{align*}
\begin{remark}
By the discussion in \cite{webb2018abelian}, the above bundles split as direct sums of associated line bundles. 
\end{remark}
Let $\alpha$ range over all the weights of $\mathbf{T}$ acting on $\mathfrak{g}$, then
\begin{align*}
\lambda^{\mathbb{C}^*} _ { -1 }( \mathrm{R}^0\pi_*(\mathfrak{P}  \times_{\mathbf{T}} \mathfrak{g})^{mov})^{\vee} 
&=\prod _{\tilde{\beta}(\alpha) \geq 0} \frac{\prod _ {k=-\infty}^{\tilde{\beta}(\alpha)}(1-L^{\vee}_{\alpha}q^k)}{\prod _{k=-\infty}^{0}(1-L^{\vee}_{\alpha}q^k)} \\
\lambda^{\mathbb{C}^*} _ { -1 }( \mathrm{R}^1\pi_*(\mathfrak{P}  \times_{\mathbf{T}} \mathfrak{g})^{mov})^{\vee} 
&=\prod _{\tilde{\beta}(\alpha) < 0} \frac{\prod _{k=-\infty}^{-1}(1-L^{\vee}_{\alpha}q^k)}{\prod _ {k=-\infty}^{\tilde{\beta}(\alpha)}(1-L^{\vee}_{\alpha}q^k)}
\end{align*}
notice that $\tilde{\beta}(\alpha) < 0$, when $\alpha \in R^+ \backslash R_L$, so
\begin{align*}
\frac{\lambda^{\mathbf{T}} _ { -1 }( \mathrm{R}^0\pi_*(\mathfrak{P} \times_{\mathbf{T}} \mathfrak{g})^{mov})^{\vee} }{\lambda^{\mathbf{T}} _ { -1 }( \mathrm{R}^1\pi_*(\mathfrak{P} \times_{\mathbf{T}} \mathfrak{g})^{mov})^{\vee} \prod _ { \alpha \in R ^ { + } \backslash R _ { L } } \left( 1 -  L^{\vee}_\alpha \right)} 
= \prod_{\alpha} \frac{\prod _ {k=-\infty}^{\tilde{\beta}(\alpha)}(1-L^{\vee}_{\alpha}q^k)}{\prod _{k=-\infty}^{0}(1-L^{\vee}_{\alpha}q^k)}
\end{align*}
as we define before, the $\mathbf{T}$ weights of V are $\{\xi_1, \cdots \xi_n \}$, suppose $I \cup J=\{1,\cdots,n \}$ such that, $\tilde{\beta}(\xi _i) \geq 0$ if $i \in I$, and $\tilde{\beta}(\xi _j) <0 $ if $j \in J$, then 
\begin{align*}
\lambda^{\mathbf{T}} _ { -1 }( \mathrm{R}^0\pi_*(\mathfrak{P} \times_{\mathbf{T}} V)^{mov})^{\vee} 
&=\prod _{i \in I} \frac{\prod _ {k=-\infty}^{\tilde{\beta}(\xi _i)}(1-L^{\vee}_{\xi _i}q^k)}{\prod _{k=-\infty}^{0}(1-L^{\vee}_{\xi _i}q^k)} \\
\lambda^{\mathbf{T}} _ { -1 }( \mathrm{R}^1\pi_*(\mathfrak{P} \times_{\mathbf{T}} V)^{mov})^{\vee} 
&=\prod _{j \in J} \frac{\prod _{k=-\infty}^{-1}(1-L^{\vee}_{\xi _j}q^k)}{\prod _ {k=-\infty}^{\tilde{\beta}(\xi _j)}(1-L^{\vee}_{\xi _j}q^k)}
\end{align*}
similarly,
\begin{align*}
  \frac { \lambda^{\mathbf{T}} _ { -1 } ( V ^ { s } \times _ \mathbf{ T } V )^{\vee} } { \lambda^{\mathbf{T}} _ { -1 } ( V ^ { s } \times _ \mathbf{ T } V _ { \tilde { \beta } } )^{\vee} }  \frac{\lambda^{\mathbf{T}} _ { -1 }( \mathrm{R}^1\pi_*(\mathfrak{P} \times_{\mathbf{T}} V )^{mov})^{\vee}}{\lambda^{\mathbf{T}} _ { -1 }( \mathrm{R}^0\pi_*(\mathfrak{P} \times_{\mathbf{T}} V )^{mov})^{\vee}} 
= \prod_{j=1}^{n} \frac{\prod _ {k=-\infty}^{0}(1-L^{\vee}_{\xi_j}q^k)}{\prod _{k=-\infty}^{\tilde{\beta}(\xi_j)}(1-L^{\vee}_{\xi_j}q^k)}
\end{align*}

Comparing formula (3) with the formula in \cite{2018arXiv180406552R}, we get 
\begin{align*}
\prod_{j=1}^{n} \frac{\prod _ {k=-\infty}^{0}(1-L^{\vee}_{\xi_j}q^k)}{\prod _{k=-\infty}^{\tilde{\beta}(\xi_j)}(1-L^{\vee}_{\xi_j}q^k)} \cdot \mathrm{det}^{-l} (\mathrm{R}^{\bullet}\pi_*(\mathcal{P} \times_{\mathbf{T}} R)) = I_{\tilde{\beta}}^{V//\mathbf{T},R,l}
\end{align*}
thus we arrive at the following theorem,
\begin{theorem}
Let $\mathbf{G}$ be a connected reductive complex Lie group with character $\theta$, acting on a vector space $V$ satisfying 
\item[-]{$V^s(\mathbf{G})=V^{ss}(\mathbf{G})$is nonempty and $V^s(\mathbf{T})=V^{ss}(\mathbf{T})$}
\item[-]{$\mathbf{G}$ acts on $V^s(\mathbf{G})$ freely and $\mathbf{T}$ acts on $V^s(\mathbf{T})$ freely}
\item[-]{The GIT quotients $V//_{\theta}\mathbf{G}=V^s(\mathbf{G})/\mathbf{G}$ and $V//_{\theta}\mathbf{T}=V^s(\mathbf{T})/\mathbf{T}$ are projective.}
Then 
\begin{align*}
\psi^*{I}_{\beta}^{V//\mathbf{G},R,l}(q,Q)
=j^* \sum_{W\tilde{\beta} \rightarrow \beta} \sum_{w \in W/W_L} w \left[ \prod_{\alpha} \frac{\prod _ {k=-\infty}^{\tilde{\beta}(\alpha)}(1-L^{\vee}_{\alpha}q^k)}{\prod _{k=-\infty}^{0}(1-L^{\vee}_{\alpha}q^k)}  I_{\tilde{\beta}}^{V//\mathbf{T},R,l} \right]
\end{align*}
where $j$ is an open immersion induced by the inclusion $ V^s(\mathbf{G}) \subset V^s(\mathbf{T})$. 
\end{theorem}
\begin{remark}
Here we can't combine $w$ with $W\tilde{\beta}$ as in \cite{webb2018abelian}, since in K theory $(1-L^{\vee}_{-\alpha}) \neq -(1-L^{\vee}_{\alpha})$, however in cohomology theory, we have $c_1(L_{\alpha})=-c_1(L_{-\alpha})$.
\end{remark}

\section{Difference equations of $I$-function of complete flag variety with level structures}

\subsection{$I$-function of partial flag varieties}

Partial flag variety is a manifold that parameterizes inclusion sequences  $V_{\bullet} : V_1 \hookrightarrow \cdots \hookrightarrow V_I $ of planes $ V_i $ in $ \mathbb{C}^n $ of dimension $ r_i $, denoted by $F l _ { r _ { 1 } , \ldots , r _ { I } } \left( \mathbb { C } ^ { n } \right)$. It is a projective variety, indeed, there is a natural map $ \psi : F l _ { r _ { 1 } , \ldots , r _ { I } } \left( \mathbb { C } ^ { n } \right) \rightarrow \operatorname { Gr } \left( r _ { 1 } , \mathbb{C}^n \right) \times \cdots \times G r \left( r _ { I } , \mathbb{C}^n \right) $ and the image is closed.

Here we describe flag variety as a GIT quotient, we claim  
\begin{align*}
F l _ { r _ { 1 } , \ldots , r _ { I } } \left( \mathbb { C } ^ { n } \right) =  \oplus^I_{i=1} \mathrm{Hom}(\mathbb{C}^{r_i},\mathbb{C}^{r_{i+1}}) //_{\mathbf{det}} GL_{r_1} \times \cdots \times GL_{r_I}
\end{align*}
here 
\begin{align*}
& \mathbf{det} \in \chi(Gl_{r_1} \times \cdots \times Gl_{r_I}) , \ \ (g_{r_1}, \cdots, g_{r_I}) \in Gl_{r_1} \times \cdots \times Gl_{r_I} , \\  & \mathbf{det}((g_{r_1}, \cdots, \ \ g_{r_I})) = \mathbf{det}(g_{r_1}) \times \cdots \times \mathbf{det}(g_{r_I})   
\end{align*}
and we set $\mathbb{C}^{r_I+1} = \mathbb{C}^n$.

Let's figure out what lies in the stable loci, if we view $ \mathrm{Hom}(\mathbb{C}^{r_i},\mathbb{C}^{r_{i+1}}) $ as matrixes $ \mathrm{M}_{r_i \times r_{i+1}} $, we claim that if $ A := (A_i)_i \in \oplus_{i=1}^{I}\mathrm{M}_{r_i \times r_{i+1}} $ with all $ \mathrm{rank}(A_i) = r_i $ then $ (A_i)_i$ belongs to stable loci. Indeed, let the one parameter subgroups $ \lambda  := (\lambda_i)_i : \mathbb{C}^* \rightarrow \prod_{i=1}^{I} GL_{r_i} $ be 
\begin{align*}
 \lambda _i(t) = P_i\operatorname{diag}(t^{d_{i,1}}, \cdots, t^{d_{i,r_i}})P^{-1}_i      
\end{align*}
such that $ \mathrm{lim}_{t \rightarrow 0} \lambda \cdot A $ exists, notice that $ P_I\operatorname{diag}(t^{d_{I,1}}, \cdots, t^{d_{I,r_I}})P^{-1}_I \cdot A_I $ exists, this requires all $d_{I,j} \geq 0 $, deductively, we can get all $d_{i,j} \geq 0 $, then $\mathbf{det}(\lambda)=t^{\sum_{i,j}d_{i,j}}$ with $\sum_{i,j}d_{i,j} \geq 0$.

Conversely, if $ (A_i)_i \in \oplus_{i=1}^{I}\mathrm{M}_{r_i \times r_{i+1}} $ with some $ \mathrm{rank}(A_i) < r_i $ then there exist $ P \in Sl_{r_i} $ and $ Q \in Sl_{r_{i+1}} $ such that $ P^{-1} A_i Q $ has the first row zero, let part of the one parameter subgroups $ (\lambda _i, \lambda_{i+1}) : \mathbb{C}^* \rightarrow GL_{r_i} \times GL_{r_{i+1}} $ be 
\begin{align*}
 \lambda _i(t) = P\operatorname{diag}(t^{d_{i1}}, \cdots, t^{d_{i,r_i}})P^{-1} \ \     
and \ \
 \lambda _{i+1}(t) = Q\operatorname{diag}(t^{d_{i+1,1}}, \cdots, t^{d_{i+1,r_i}})Q^{-1}      
\end{align*}
then
\begin{align*} 
(\lambda _i, \lambda_{i+1}) \cdot A_i = P\operatorname{diag}(t^{d_{i,1}}, \cdots, t^{d_{i,r_i}})P^{-1} A Q^{-1}\operatorname{diag}(t^{d_{i+1,1}}, \cdots, t^{d_{i+1,r_i}})Q  
\end{align*}
so $(\lambda _i, \lambda_{i+1}) \cdot A_i$ exists, even $d_{i,1} \ll 0 $. Thus, $(A_i)_i$ with some $ \mathrm{rank}(A_i) < r_i $ is unstable. 

Conclusion, $(\oplus^I_{i=1} \mathrm{Hom}(\mathbb{C}^{r_i},\mathbb{C}^{r_{i+1}}) )^s$ is $I$-tuple $ (f_{r_1}, \cdots, f_{r_I}) $ of injective maps $ f_{r_k} \in \mathrm{Hom}(\mathbb{C}^{r_k},\mathbb{C}^{r_{k+1}}) $, these maps give a sequence of inclusion in $ \mathbb{C}^n $, then quotient by the group $ \mathbf{G} =GL_{r_1} \times \cdots \times GL_{r_I} $ which is the transformation matrixes of changing basis, so we have a inclusion sequences $V_{\bullet} : V_1 \hookrightarrow \cdots \hookrightarrow V_I $ of planes $ V_i $ in $ \mathbb{C}^n $ of dimension $ r_i $.

Recall the facts that every principal $ Gl_k $-bundle is the frame bundle of some vector bundle (up to isomorphism)[\cite{mitchell2001notes}, Proposition 4.1], and Grothendieck classification of vector bundles on $ \mathbb{P}^1 $, so $ \mathcal{P} \times_{GL_k} \mathbb{C}^k \cong \oplus^{k}_{i=1} \mathcal{O}_{\mathbb{P}^1}(-a_{k,i}) $, denoted by ${\mathcal{E}}_k $, then we have
\begin{align*}
\mathcal{P} \times _{\mathbf{G}} V \cong \oplus ^I _{i=1} \mathcal{H}om( {\mathcal{E}}_{r_i}, {\mathcal{E}}_{r_{i+1}} )
\end{align*}
and a section $ u : \mathbb{P}^1 \rightarrow \oplus ^I _{i=1} \mathcal{H}om( {\mathcal{E}}_{r_i}, {\mathcal{E}}_{r_{i+1}} ) $. The condition requiring no base point for $ s \neq 0 \in \mathbb{P}^1 $ means $ {\mathcal{E}}_{r_i} $ is a $ \mathbb{C}^*$ invariant coherent sheaf. By definition, there is a filtration of $ {\mathcal{E}}_{r_i} $, denoted by $ F_{\bullet}{\mathcal{E}}_{r_i} $: 
\begin{align*}
{\mathcal{E}}_{r_i,1} \subset {\mathcal{E}}_{r_i,2} \subset \cdots \subset {\mathcal{E}}_{r_i,r_i}
\end{align*}
such that ${\mathcal{E}}_{r_i,j}/{\mathcal{E}}_{r_i,j-1} \cong \mathcal{O}_{\mathbb{P}^1}(-a_{r_i,j})$.

Recall the $ \mathbb{C}^* $ invariant coherent sheaf on $ \mathbb{P}^1 $ is characterized as follow
\begin{proposition}\cite{liu2004s}
Let $\mathcal{V}$ be a rank $r$ coherent subsheaf of $\mathcal{E}^n$ on $\mathbb{P}^1$. Then $\mathcal{V}$ is a locally free $ \mathcal{O}_{\mathbb{P}^1} $-module.

Let $ \left\{ U _ { 0 } = \mathbb{P}^1 - \{ \infty \} = \operatorname { Spec } \mathbb { C } [ z ] , U _ { \infty } = \mathbb{P}^1 - \{ 0 \} = \operatorname { Spec } \mathbb { C } [ w ] \right\} $ be an atlas of affine charts on $ \mathbb{P}^1 $ . Then there exists a constant re-trivialization
\begin{align*}
\mathcal{E}^n |_{U_0} = \left( \mathbb { C } [ z ] ^ { \oplus r } \right) ^ { \sim }
\end{align*}
such that 
\begin{align*}
\left. \mathcal { V } \right| _ { U _ { 0 } } = \left( \mathbb { C } [ z ] z ^ { \alpha _ { 1 } } \oplus \cdots \oplus \mathbb { C } [ z ] z ^ { \alpha _ { r } } \right) ^ { \sim }
\end{align*}
with $ 0 \leq \alpha _ { 1 } \leq \cdots \leq \alpha _ { r } $, Similarly for $ \left. \mathcal { V } \right| _ { U _ { \infty } } $ corresponding to $0 \leq \beta _ { 1 } \leq \ldots , \leq \beta _ { r }$.
\end{proposition}
\begin{definition}{[admissible pair of $ (\alpha_{\bullet},\beta_{\bullet} ) $]}
Let $ (\alpha_{1, \bullet},\beta_{1, \bullet} ) $ and $ (\alpha_{2, \bullet},\beta_{2, \bullet} ) $  be two datum associated with sheaves $ (\alpha_{1,1}, \cdots, \alpha_{1,r_1} ; \beta_{1,1}, \cdots, \beta_{1,r_1} ) $ and $ (\alpha_{2,1}, \cdots, \alpha_{2,r_2} ; \beta_{2,1}, \cdots, \beta_{2,r_2} ) $, we say  $ (\alpha_{1, \bullet},\beta_{1, \bullet} ) $ is admissible to $ (\alpha_{2, \bullet},\beta_{2, \bullet} ) $ if
\begin{align*}
r _ { 1 } \leq r _ { 2 } , \quad \alpha _ { 1 , i } \geq \alpha _ { 2 , i } \quad \text { and } \quad \beta _ { 1 , i } \geq \beta _ { 2 , i } \text { for } i = 1 , \ldots , r _ { 1 }
\end{align*}
denoted by $ (\alpha_{1, \bullet},\beta_{1, \bullet} ) \rightarrow  (\alpha_{2, \bullet},\beta_{2, \bullet} ) $.
\end{definition}
\begin{proposition}[$\mathbb{C}^*$-invariant pair]\cite{liu2004s}
Let $\mathcal{V}_1$ and $\mathcal{V}_2$ be $\mathbb{C}^*$-invariant subsheaves of $\mathcal{E}^n$ characterized by the data $ (\alpha_{1, \bullet},\beta_{1, \bullet} ) $ and $ (\alpha_{2, \bullet},\beta_{2, \bullet} ) $ respectively. Then $\mathcal{V}_1$ is a subsheaf of $\mathcal{V}_2$ if and only if $ (\alpha_{1, \bullet},\beta_{1, \bullet} ) $ is admissible to $ (\alpha_{2, \bullet},\beta_{2, \bullet} ) $.
\end{proposition}

Since $ {\mathcal{E}}_{r_i} $ is characterized by $ (a_{{r_i},\bullet }) $, then we claim that $  \mathrm{Hom} ({\mathcal{E}}_ {r_i} , {\mathcal{E}}_{r_{i+1}} ) \neq \emptyset $ if and only if $ (a_{{r_i},\bullet }) \rightarrow (a_{{r_{i+1}},\bullet }) $. Thus, the sheaves ${\mathcal{E}}_{\bullet} $ are characterized by following matrix
\begin{align}
A = \left[ \begin{array} { c c c c c } { a _ { r_1,1 } } & {\cdots } & {a _ { r_1,r_1 }} & { } & { } \\ { \vdots } & { \ddots } & { \ddots } & { \ddots } & { } \\ 
{ a _ { r_I , 1 } } & { \cdots } & { \cdots } & { \cdots } & { a _ { r_I , r_I } } \end{array} \right] _ { I \times r_I }
\end{align}
such that $ (\alpha_{r_1,\bullet}) \rightarrow (\alpha_{r_2,\bullet}) \rightarrow \cdots \rightarrow (\alpha_{r_I,\bullet}) $. And each such a matrix A will give a connected component of $ \mathrm{F}_{0,\beta} $, denoted by $ \mathrm{F}_{A}$. There are tautological sheaves $\widetilde{\mathcal{S}}_{r_i,\bullet}$ on $F_A$ determined by data $(a_{r_i,1},\cdots,a_{r_i,r_i})$, for more details, see \cite{liu2004s}. By proposition 3.1, we know $F_A$ can be identified as $ V^s_A/\mathbf{P}_A $ where $V^s_A$ is a closed subspace of $V^s(\mathbf{G})$ and $\mathbf{P}_A$ is a parabolic subgroup of $\mathbf{G}$, and both are determined by $A$.

More explicitly, from the discussion below (1), the long exact sequence (2) can be realized as
\begin{align*}
0 \rightarrow \oplus ^I _{i=1} \mathcal{H}om( {\mathcal{S}}_{r_i}, {\mathcal{S}}_{r_{i+1}} ) \rightarrow \oplus ^I _{i=1} \mathcal{H}om( {\mathcal{S}}_{r_i}, {\mathcal{S}}_{r_{i+1}} ) \rightarrow \mathcal{F} |_{\mathrm{F}_{A}} \rightarrow 0
\end{align*}
on the universal curve of $ \mathrm{F}_{A} $, 
\begin{align*}
\pi _1 : \mathrm{F}_{A} \times \mathbb{P}^1 \rightarrow  \mathrm{F}_{A} 
\end{align*}
where $S_{r_i}$ has a filtration ${F}_{\bullet}S_{r_i}$ such that
\begin{align*}
{\mathcal{S}}_{r,j} / {\mathcal{S}}_{r,j-1} = \left( \pi _ { 1 } ^ { * } \left(  { \widetilde{\mathcal { S }} } _ { i , j } /  { \widetilde{\mathcal { S }} } _ { i , j - 1 } \right) \right) \left( - a _ { i , j } z \right) \end{align*}
then we have 
\begin{align}
\psi^* I(q,Q) &= 1 + \sum_{ \{ d_{i}=\sum ^{r_i}_l a_{r_i, l} \}  }Q_{1}^{d_1} \cdots Q_{I}^{d_I} \sum_{w \in W/W_L } w  \{  \frac{\prod_{i, j < j^{\prime}}\prod_{k=1}^{a _ { r_i , j ^ { \prime } } - a _ {r_i , j } }(1-l^{\vee}_{r_i, j}l_{r_i, j^{\prime}}q^k)}{\prod_{i, j , j^{\prime \prime}}\prod_{k=1}^{a _ { r_{i+1} , j ^ { \prime \prime } } - a _ { r_i , j } }(1-l^{\vee}_{r_i, j}l_{r_{i+1},j^{\prime \prime }}q^k) } \nonumber \\ 
       & \cdot \frac{\prod_{i, j , j^{\prime \prime}}\prod_{k=0}^{a _ { r_{i+1} , j ^ { \prime \prime } } - a _ { r_i , j } -1}(1-l^{\vee}_{r_i, j}l_{r_{i+1},j^{\prime \prime }}q^{-k})}{\prod_{i, j < j^{\prime}}\prod_{k=1}^{a _ { r_i , j ^ { \prime } } - a _ { r_i , j } -1}(1-l^{\vee}_{r_i, j}l_{r_i, j^{\prime}}q^{-k})}  \frac{1}{\prod_{i, j < j^{\prime}}(1-l^{\vee}_{r_i, j}l_{r_i, j^{\prime}}) }  \}
\end{align}
where we set $l_{i,j}$ to be $\bar{\psi}^* (\widetilde{\mathcal{S}}_{r,j} / \widetilde{\mathcal{S}}_{r,j-1})^{\vee} $, and $\{ d_{i,\bullet} \}=\{ a_{r_i,\bullet} \}$ are admissible partitions.

\begin{remark}
For the complete flag variety, the map $i: V^s_{\tilde{\beta}}/\mathbf{P}_{\tilde{\beta}} \rightarrow  V^s(\mathbf{G})  / \mathbf{P} _ { \tilde { \beta } } $ is trivial, the evaluation map $ev_{\bullet}$ is identity.
\end{remark}

\begin{example}
Let's compute the degree 1 term of the $I$-function for complete flag variety $\mathrm{Fl}(\mathbb{C}^4)$, there are three fixed loci component
\begin{align*}
A = \left[ \begin{array} { c c c  } { 1 } & { } & { }  \\ { 0 } & { 0 } & {  }  \\ { 0 } & { 0 } & { 0 }  \end{array} \right], \ \ B= \left[ \begin{array} { c c c  } { 0 } & { } & { }  \\ { 0 } & { 1 } & {  }  \\ { 0 } & { 0 } & { 0 }  \end{array} \right], \ \ C= \left[ \begin{array} { c c c  } { 0 } & { } & { }  \\ { 0 } & { 0 } & {  }  \\ { 0 } & { 0 } & { 1 }  \end{array} \right]
\end{align*}
Let $ 0 \subset \mathcal{S}_1 \subset \mathcal{S}_2 \subset \mathcal{S}_3 \subset \mathbb{C}^4 \otimes \mathcal{O}_{Fl} $ be the tautological bundles on $\mathrm{Fl}(\mathbb{C}^4)$ and set, 
\begin{align*}
l_1= \mathcal{S}^{\vee}_1 ,\ \ l_2 = (\mathcal{S}_2 / \mathcal{S}_1)^{\vee} ,\ \ l_3 = (\mathcal{S}_3 / \mathcal{S}_2 )^{\vee}
\end{align*}
then the contributions from three of them are
\begin{align*}
 A : \frac{Q_1}{(1-q)(1-l_1 l_2^{\vee}q)} , \ \
 B : \frac{Q_2}{(1-q)(1-l_2 l_3^{\vee}q)} , \ \
 C : \frac{Q_3}{(1-q)(1-l_1 l_2 l_3^{2}q)}
\end{align*}
by Pl\"ucker embedding $\mathrm{Fl}(\mathbb{C}^4) \hookrightarrow \prod^3_{i=1} \mathbb{P}(\Lambda^{i} \mathbb{C}^i \subset \Lambda^i \mathbb{C}^4 ) $, set $ p_i $ are pull back of $ \mathcal{O}(1) $ of the i-th $ \mathbb{P}^3 $, then we have 
\begin{align*}
l_1 = p_1 ,\ \ l_2 = p^{-1}_1 p_2,\ \ l_3 = p^{-1}_2 p_3,\ \ l_1 l_2 l_3 = p_3 
\end{align*}
substitute into the contributions, then
\begin{align*}
 A : \frac{Q_1}{(1-q)(1-p^2_1 p^{-1}_2q)} , \ \
 B : \frac{Q_2}{(1-q)(1-p^{-1}_1 p^2_2 p^{-1}_3 q)} , \ \
 C : \frac{Q_3}{(1-q)(1-p^{-1}_2 p_3^{2}q)}
\end{align*}
these agree with the result of \cite{givental2003quantum}. 
\end{example}

\subsection{Difference equation}

In \cite{givental2003quantum}, Givental and Y.P.Lee found the difference equation of small $J$-function of complete flag variety, $Fl_r(\mathbb{C}^{r+1})$, we can get some similar difference equations of $I$-function with different levels, more precisely, using the notation of \cite{iritani2014reconstruction}
\begin{equation*}
\tilde{\mathcal{J}}(Q,q)=\prod^{r}_{i=1}p_{i}^{\frac{lnQ_{i}}{lnq}}\sum_{d_1 + \cdots + d_r \geq 0 }\mathcal{J}_{\mathbf{d}}
\end{equation*}
where
\begin{align*}
\mathcal{J}_{\mathbf{d}}=\sum _i \phi _i  \left\langle \frac{\phi^i}{1-qL} \right\rangle _{0,1,d}Q^{d_1}_1 \cdot Q^{d_2}_2 \cdots Q^{d_r}_r
\end{align*}
for $\mathbf{d}=(d_1,d_2, \cdots,d_r)$, $d=d_1+\cdots+d_r$. From \cite{givental2003quantum}, we know that $\tilde{J}$-function satisfies the following difference equation
\begin{equation*}
\hat { H } _ { Q , q } \tilde{\mathcal{J}}(Q,q)=(\Lambda _ { 0 } ^ { - 1 } + \ldots + \Lambda _ { r } ^ { - 1 }) \tilde{\mathcal{J}}(Q,q)
\end{equation*}
where we consider the equivariant K-theory, and
\begin{equation*}
\hat { H } _ { Q , q } \coloneq q ^ {  Q_1  \partial_{ Q _ { 1 } } }+ q ^ { Q_2 \partial_{ Q _ { 2 } }- Q_1 \partial_{ Q _ { 1 } }} \left( 1 - { Q_ 1 } \right) + \ldots + q ^ { - Q_r \partial_{ Q _ { r } }} \left( 1 - { Q_r } \right)
\end{equation*}
then we have the following recursive formula
\begin{equation}
\begin{split}
&\left[ p  _ { 1 } \left( q ^ { d _ { 1 } } - 1 \right)  + \ldots + p _ { i } p ^ { -1 } _ { i - 1 }  \left( q ^ { d _ { i } - d _ { i - 1 } } - 1 \right) + \ldots + p ^ { -1 } _ { r } \left( q ^ { - d _ { r } } - 1 \right) \right] \tilde{\mathcal{J}} _ {\mathbf{d} } \\
&=p _ { 2 } p ^ { -1 } _ { 1 } q ^ { d _ { 2 } - d _ { 1 } } \tilde{\mathcal{J}} _ { \mathbf{d} - \mathbf { 1 } _ { 1 } } + \ldots + p ^ { -1 } _ { r }  q ^ { - d _ { r } } \tilde{\mathcal{J}} _ { \mathbf{d} - \mathbf { 1 } _ { r } }
\end{split}
\end{equation} 
Let's consider representations given by the following characters of $ \mathbf{G}= GL_1 \times GL_2 \times \cdots \times GL_r   $
\begin{align*}
 \theta_i :  \mathbf{G} \rightarrow \mathbb{C}^* , \ \
            (g_1, \cdots, g_r) \mapsto \mathrm{det}(g_i)         
\end{align*}
then the associated level $ l_i $ structure is 
\begin{align*}
\mathcal { D } ^ {\theta_i, l_i } : = \operatorname { det } ^ { - l_i } R^{\bullet} \pi _ { * } \left( \mathfrak { P } \times _ { \mathbf{G} } \mathbb{C}_{\theta_i} \right)
\end{align*}
by previous discussion, we have
\begin{equation*}
\mathfrak { P } \times _ { \mathbf{G} } \mathbb{C}_{\theta_i} = \wedge^i \left( \pi^* {\mathcal{S}}_i \right) \otimes \mathcal{O}(-d_i)
\end{equation*}
then the level $l_i$ determinant bundle of pushing down to Hyper-Quot scheme and restrict to the distinguish fixed locus is
\begin{equation*}
 \mathcal { D } ^ {\theta_i, l_i }|_{\mathrm{F}_{0,\beta}} =p_i^{ l_i (d_i-1)}q^{ l_i  \frac{d_i(d_i-1)}{2}}
\end{equation*}
then the new I-function with level is
\begin{align*}
\tilde{I} ^ {\theta_i, l_i }:=\tilde{J}^{\theta_i, l_i ,0^+}= \prod^{r}_{i=1}p_{i}^{\frac{lnQ_{i}}{lnq}} \sum_{d_1 + \cdots + d_r \geq 0 }{I}^{\theta_i,l_i}_{\mathbf{d}}
\end{align*}
where
\begin{align*}
{I}^{\theta_i, l_i }_{\mathbf{d}} = I_\mathbf{d} \cdot p_i^{ l_i (d_i-1)}q^{ l_i  \frac{d_i(d_i-1)}{2}}
\end{align*}
and $\mathbf{d}=(d_1,d_2, \cdots,d_r)$. Using (6) we can get the new $I$-function satisfies the following equation
\begin{equation*}
\widetilde { H }^{\theta_i,l_i} _ { Q , q }  \tilde{\mathcal{J}}^{\theta_i,l_i,0^+}(Q,q)=(\Lambda _ { 0 } ^ { - 1 } + \ldots + \Lambda _ { r } ^ { - 1 }) \tilde{\mathcal{J}}^{R,l,0^+}(Q,q)
\end{equation*}
where
\begin{equation*}
\widetilde { H }^{\theta_i,l_i} _ { Q , q } \coloneq q ^ {  Q_1  \partial_{ Q _ { 1 } } }+ \cdots + q ^ { Q_{i+1} \partial_{ Q _ { i+1 } } - Q_{i} \partial_{ Q _ { i } }} \left( 1 - { Q_ i } \circ q^{l_i Q_i \partial_{Q_i}} \right) + \ldots + q ^ { - Q_r \partial_{ Q _ { r } }} \left( 1 - { Q_r } \right)
\end{equation*}
and
\begin{align*}
\tilde{\mathcal{J}}^{\theta_i,l_i,0^+}(Q,q) = \prod^{r}_{i=1}p_{i}^{\frac{lnQ_i}{lnq}}\sum_{d_1 + \cdots + d_r \geq 0 }J^{\theta_i,l_i,0^+}_{\mathbf{d}}
\end{align*}
so we get the following corollary.
\begin{corollary}
The $K_{\mathbf{T}}(X)$-valued vector-series $\tilde{\mathcal{J}}^{\theta_i,l_i,0^+}(Q,q)$ is the eigen-vector of the finite-difference operator $\widetilde{H}^{\theta_i,l_i}_{Q,q}$ with the eigen-value $\Lambda _ { 0 } ^ { - 1 } + \ldots + \Lambda _ { r } ^ { - 1 }$.
\end{corollary}

\bibliographystyle{plain}
\bibliography{ref7}
\end{document}